\documentclass[10pt,oneside]{amsart}

\usepackage[
paper=a4paper,
text={137mm,205mm},centering
]{geometry}

\usepackage{amssymb}
\usepackage[all]{xy}
\usepackage{graphicx,color,float,array}
\usepackage[bookmarks]{hyperref}
\usepackage{pinlabel}
\usepackage[  
  textwidth=1.1in,
  backgroundcolor=yellow,
  bordercolor=orange,
  textsize=footnotesize
]{todonotes}
\usepackage{pb-diagram,pb-xy}
\usepackage{booktabs,longtable}

\iftrue
\makeatletter
\def\@settitle{%
  \vspace*{-20pt}
  \begin{flushleft}%
    \baselineskip14\p@\relax
    \normalfont\bfseries\LARGE
    \@title
  \end{flushleft}%
}
\def\@setauthors{%
  \begingroup
  \def\thanks{\protect\thanks@warning}%
  \trivlist
  \large \@topsep30\p@\relax
  \advance\@topsep by -\baselineskip
  \item\relax
  \author@andify\authors
  \def\\{\protect\linebreak}%
  \authors
  \ifx\@empty\contribs
  \else
    ,\penalty-3 \space \@setcontribs
    \@closetoccontribs
  \fi
  \normalfont
  \endtrivlist
  \endgroup
}
\def\@setaddresses{\par
  \nobreak \begingroup\raggedright
  \small
  \def\author##1{\nobreak\addvspace\smallskipamount}%
  \def\\{\unskip, \ignorespaces}%
  \interlinepenalty\@M
  \def\address##1##2{\begingroup
    \par\addvspace\bigskipamount\noindent
    \@ifnotempty{##1}{(\ignorespaces##1\unskip) }%
    {\ignorespaces##2}\par\endgroup}%
  \def\curraddr##1##2{\begingroup
    \@ifnotempty{##2}{\nobreak\noindent\curraddrname
      \@ifnotempty{##1}{, \ignorespaces##1\unskip}\/:\space
      ##2\par}\endgroup}%
  \def\email##1##2{\begingroup
    \@ifnotempty{##2}{\smallskip\nobreak\noindent E-mail address%
      \@ifnotempty{##1}{, \ignorespaces##1\unskip}\/:\space
      \ttfamily##2\par}\endgroup}%
  \def\urladdr##1##2{\begingroup
    \def~{\char`\~}%
    \@ifnotempty{##2}{\nobreak\noindent\urladdrname
      \@ifnotempty{##1}{, \ignorespaces##1\unskip}\/:\space
      \ttfamily##2\par}\endgroup}%
  \addresses
  \endgroup
  \global\let\addresses=\@empty
}
\def\@setabstracta{%
    \ifvoid\abstractbox
  \else
    \skip@25\p@ \advance\skip@-\lastskip
    \advance\skip@-\baselineskip \vskip\skip@
    \box\abstractbox
    \prevdepth\z@ 
    \vskip-10pt
  \fi
}
\renewenvironment{abstract}{%
  \ifx\maketitle\relax
    \ClassWarning{\@classname}{Abstract should precede
      \protect\maketitle\space in AMS document classes; reported}%
  \fi
  \global\setbox\abstractbox=\vtop \bgroup
    \normalfont\small
    \list{}{\labelwidth\z@
      \leftmargin0pc \rightmargin\leftmargin
      \listparindent\normalparindent \itemindent\z@
      \parsep\z@ \@plus\p@
      
    }%
    \item[\hskip\labelsep\bfseries\abstractname.]%
}{%
  \endlist\egroup
  \ifx\@setabstract\relax \@setabstracta \fi
}
\def\section{\@startsection{section}{1}%
  \z@{-1.2\linespacing\@plus-.5\linespacing}{.8\linespacing}%
  {\normalfont\bfseries\large}}
\def\subsection{\@startsection{subsection}{2}%
  \z@{-.8\linespacing\@plus-.3\linespacing}{.3\linespacing\@plus.2\linespacing}%
  {\normalfont\bfseries}}
\def\subsubsection{\@startsection{subsubsection}{3}%
  \z@{.7\linespacing\@plus.1\linespacing}{-1.5ex}%
  {\normalfont\itshape}}
\def\@secnumfont{\bfseries}
\makeatother
\fi 

\def\to{\mathchoice{\longrightarrow}{\rightarrow}{\rightarrow}{\rightarrow}}
\makeatletter
\newcommand{\shortxra}[2][]{\ext@arrow 0359\rightarrowfill@{#1}{#2}}
\def\longrightarrowfill@{\arrowfill@\relbar\relbar\longrightarrow}
\newcommand{\longxra}[2][]{\ext@arrow 0359\longrightarrowfill@{#1}{#2}}
\renewcommand{\xrightarrow}[2][]{\mathchoice{\longxra[#1]{#2}}%
  {\shortxra[#1]{#2}}{\shortxra[#1]{#2}}{\shortxra[#1]{#2}}}
\makeatother

\makeatletter
\def\Nopagebreak{\@nobreaktrue\nopagebreak}
\makeatother

\theoremstyle{plain}
\newtheorem{theorem}{Theorem}[section]
\newtheorem{proposition}[theorem]{Proposition}
\newtheorem{corollary}[theorem]{Corollary}
\newtheorem{lemma}[theorem]{Lemma}

\theoremstyle{definition}

\newtheorem*{question}{Question}

\newtheorem{remark}[theorem]{Remark}

\def\Z{\mathbb{Z}}

\def\Ker{\operatorname{Ker}}

\def\Im{\operatorname{Im}}

\def\trace{\operatorname{tr}}

\def\surj{\mathbin{\hbox to 0mm{$\rightarrow$\hss}\kern.5ex\hbox{$\rightarrow$}}}

\def\SL{\operatorname{SL}}

\def\sbmatrix#1{\big[\begin{smallmatrix}#1\end{smallmatrix}\big]}

\begin{document}

\title
{Non-meridional epimorphisms of knot groups}

\author{Jae Choon Cha}

\address{Department of Mathematics\\
  POSTECH \\
  Pohang 790--784\\
  Republic of Korea
  \quad-- and --\linebreak
  School of Mathematics\\
  Korea Institute for Advanced Study \\
  Seoul 130--722\\
  Republic of Korea}
\email{jccha@postech.ac.kr}

\author{Masaaki Suzuki}

\address{Department of Frontier Media Science\\
Meiji University\\
4--21--1 Nakano\\
Tokyo 164--8525\\
Japan
}
\email{macky@fms.meiji.ac.jp}

\def\subjclassname{\textup{2010} Mathematics Subject Classification}
\expandafter\let\csname subjclassname@1991\endcsname=\subjclassname
\expandafter\let\csname subjclassname@2000\endcsname=\subjclassname


\begin{abstract}
  In the literature of the study of knot group epimorphisms, the
  existence of an epimorphism between two given knot groups is mostly
  (if not always) shown by giving an epimorphism which preserves
  meridians.  A natural question arises: is there an epimorphism
  preserving meridians whenever a knot group is a homomorphic image of
  another?  We answer in the negative by presenting infinitely many
  pairs of prime knot groups $(G,G')$ such that $G'$ is a homomorphic
  image of $G$ but no epimorphism of $G$ onto $G'$ preserves
  meridians.
\end{abstract}

\maketitle


\section{Introduction}

For a knot $K$ in $S^3$, its \emph{knot group} $G(K)$ is defined by
$G(K)=\pi_1(S^3-K)$.  The study of knot groups has long history, from
the beginning of modern knot theory.  In particular, recently,
\emph{epimorphisms} of knot groups have been receiving much attention.
A key problem is to determine when there is an epimorphism between two
knot groups.  For prime knots, which are of the most interest, it is
well known that a partial order $\ge$ is obtained by defining $K\ge
K'$ if there is an epimorphism $G(K)\to G(K')$ (see, for instance,
\cite[p.~422]{Ohtsuki-Riley-Sakuma:2008-1}).

There is a fair amount of recent work on this in the literature.  In
their remarkable work~\cite{Agol-Liu:2012-1}, Agol and Liu proved a
long-standing conjecture of Simon that a knot group surjects onto only
finitely many knot groups.  It follows that for any prime knot $K$,
there are only finitely many prime knots less than or equal to~$K$.
Together with Kitano, Horie, and Matsumoto, the second author
investigated pairs of prime knots with $11$ crossings or less whose
knot groups admit epimorphisms~\cite{Kitano-Suzuki:2005-1,
  Horie-Kitano-Matsumoto-Suzuki:2011-1, Kitano-Suzuki-Wada:2005-1,
  Kitano-Suzuki-Wada:2011-1}.  In particular, they constructed many
explicit examples of epimorphisms between knot groups.
Gonzal\'ez-Ac\~una and Ram\'inez studied which knot groups
(particulary those of 2-bridge knots) admit epimorphisms onto torus
knot groups~\cite{Gonzalez-Acuna-Ramirez:2001-1,
  Gonzalez-Acuna-Ramirez:2003-1}.  In work of Ohtsuki, Riley, and
Sakuma~\cite{Ohtsuki-Riley-Sakuma:2008-1}, Hoste and
Shanahan\cite{Hoste-Shanahan:2010-1}, and Lee and
Sakuma~\cite{Lee-Sakuma:2012-1}, systematic constructions of
epimorphisms between $2$-bridge knot (and link) groups were presented
and studied.  In~\cite{Silver-Whitten:2006-1, Silver-Whitten:2008-1},
Silver and Whitten studied knot group epimorphisms preserving
peripheral structure; in particular they showed that such epimorphisms
give rise to a partial order on the set of all knots.

\subsection*{Meridional epimorphisms}

Interestingly, most (if not all) results in the literature that a knot
is less than or equal to another are shown by presenting an
epimorphism which preserves meridians.  To be more precise, we use the
following terms: we call an element $[\alpha]\in G(K)$ a
\emph{meridian} if $\alpha$ is freely homotopic to a meridian curve
lying on the boundary of a tubular neighborhood of~$K$, and we say
that a homomorphism $G(K)\to G(K')$ is \emph{meridional} if a meridian
in $G(K)$ is sent to a meridian in~$G(K')$.  In this paper knots are
unoriented, so that a meridian may be endowed with any orientation.

The following natural question arises:

\begin{question}
  \label{question:main-question}
  Is there a meridional epimorphism $G(K)\to G(K')$ whenever there is
  an epimorphism $G(K)\to G(K')$?
\end{question}

We remark that it does not ask whether all knot group epimorphisms are
meridional; it is known that there exist non-meridional epimorphisms
of knot groups.  For instance see work of Johnson and
Livingston~\cite{Johnson-Livingston:1989-1}.  We also remark that
meridional epimorphisms can be related to geometric properties, for
example, periods of knots and degree one maps between knot exteriors.
See~\cite{Kitano-Suzuki:2008-1} for details.

There are several results supporting an affirmative answer to the
above question.  For any previously known example of a knot group
$G(K)$ which admits an epimorphism onto $G(K')$, there exists a
meridional epimorphism of $G(K)$ onto $G(K')$.  In particular, all the
epimorphisms found in ~\cite{Kitano-Suzuki:2005-1,
  Horie-Kitano-Matsumoto-Suzuki:2011-1, Kitano-Suzuki-Wada:2005-1,
  Kitano-Suzuki-Wada:2011-1} for groups of prime knots with 11 or less
crossings are meridional.  For torus knot groups, there is a
meridional epimorphism whenever there is an
epimorphism~\cite{Silver-Whitten:2008-1}.  Also, the knot group
epimorphisms in \cite{Gonzalez-Acuna-Ramirez:2001-1,
  Gonzalez-Acuna-Ramirez:2003-1, Ohtsuki-Riley-Sakuma:2008-1,
  Lee-Sakuma:2012-1} are all meridional.  An epimorphism between
nontrivial knot groups preserving peripheral structure in the sense of
~\cite{Silver-Whitten:2006-1, Silver-Whitten:2008-1} are known to be
meridional \cite[Proof of Theorem~4.1]{Silver-Whitten:2006-1},
\cite[Theorem~2.1]{Hoste-Shanahan:2010-1}.  We also remark that
epimorphisms preserving peripheral structure, particularly meridional
epimorphisms of prime knot groups, can be studied via maps of
3-manifolds with well-defined degree.


Our main result is, nevertheless, that the answer is in the negative.

\begin{theorem}
  \label{theorem:main-introduction}
  There are infinitely many distinct pairs of prime knots $(K,K')$ for
  which there is an epimorphism of $G(K)$ onto $G(K')$ but there is no
  meridional epimorphism of $G(K)$ onto $G(K')$.
\end{theorem}

The proof of Theorem~\ref{theorem:main-introduction} proceeds as
follows.  In Section~\ref{section:construction-non-meridional-epi}, we
give constructions of pairs of knot groups which admit non-meridional
epimorphisms.  In Section~\ref{section:twisted-alexander-polynomial},
we detect the non-existence of meridional epimorphisms using twisted
Alexander polynomials.  This method allows us to obtain finitely many
(in fact two) ``seed'' examples.  We remark that it depends on heavy
computation infeasible by hand, and hence it seems unable to detect
infinitely many cases in this way.  In
Section~\ref{section:satellite-construction}, we present a geometric
method to produce, from the seed examples, infinitely many pairs of
knot groups which admit non-meridional epimorphisms but do not admit
meridional epimorphisms.  In
Appendix~\ref{section:tables-of-polynomials}, we present computational
results of certain twisted Alexander polynomials which are used to
prove the non-existence of a meridional epimorphism.

\subsection*{Acknowledgements}

The first author was partially supported by NRF grants 2013067043 and
2013053914. 
The second author was partially supported by KAKENHI (No. 24740035), 
Japan Society for the Promotion of Science, Japan.

\section{Construction of non-meridional knot group epimorphisms}
\label{section:construction-non-meridional-epi}

In this section we give certain explicit examples non-meridional
epimorphisms of knot groups, for some of which we will show the
non-existence of meridional epimorphisms in the next section.


\subsection{First example on the trefoil knot}
\label{subsection:first-example}

In this subsection we describe the first successful example of a pair
of knots satisfying Theorem~\ref{theorem:main-introduction}, which we
indeed found by ad-hoc trial and error attempts aided by a computer.

\begin{figure}[H]
  \labellist
  \small\hair 0mm
  \pinlabel {$x_1$} at 35 132
  \endlabellist
  \large $K_T = \vcenter{\hbox{\includegraphics[scale=.9]{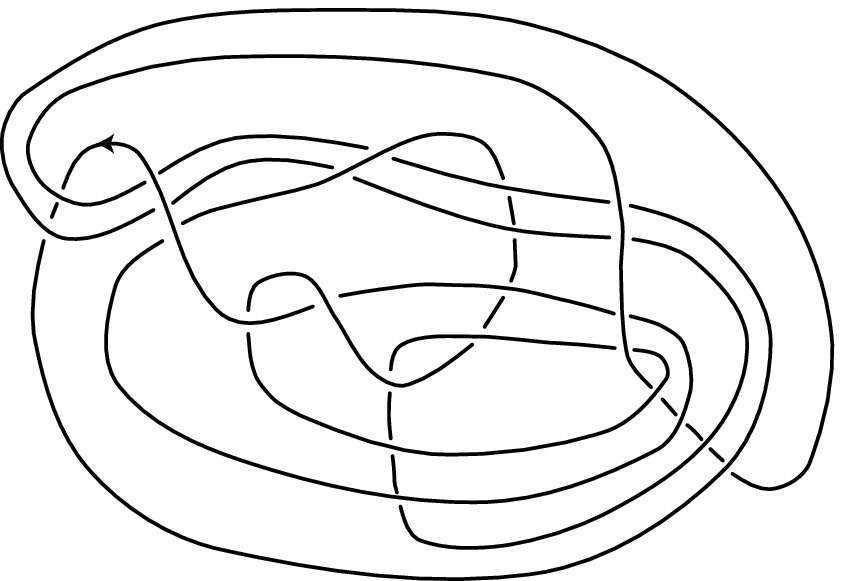}}}$
  \caption{A knot $K_T$.}
  \label{figure:example-K_T}
\end{figure}

Consider the knot $K_T$ shown in Figure~\ref{figure:example-K_T}.  The
Wirtinger presentation of the knot group $G(K_T)$ obtained from
Figure~\ref{figure:example-K_T} has 24 generators
$x_1,x_2,\ldots,x_{24}$, and 24 relators
{\allowdisplaybreaks\[
\begin{array}{llll}
  x_{6} x_{2} \bar{x}_{6} \bar{x}_{1}, & x_{10} x_{2} \bar{x}_{10} \bar{x}_{3}, 
  & x_{6} x_{3} \bar{x}_{6} \bar{x}_{4}, & x_{22} x_{4} \bar{x}_{22} \bar{x}_{5}, \\
  x_{1} x_{6} \bar{x}_{1} \bar{x}_{5}, & x_{17} x_{7} \bar{x}_{17} \bar{x}_{6}, 
  & x_{23} x_{7} \bar{x}_{23} \bar{x}_{8}, & x_{13} x_{9} \bar{x}_{13} \bar{x}_{8}, \\
  x_{3} x_{9} \bar{x}_{3} \bar{x}_{10}, & x_{1} x_{10} \bar{x}_{1} \bar{x}_{11}, 
  & x_{22} x_{12} \bar{x}_{22} \bar{x}_{11}, & x_{6} x_{13} \bar{x}_{6} \bar{x}_{12}, \\
  x_{23} x_{14} \bar{x}_{23} \bar{x}_{13}, & x_{17} x_{14} \bar{x}_{17} \bar{x}_{15}, 
  & x_{18} x_{16} \bar{x}_{18} \bar{x}_{15}, & x_{6} x_{17} \bar{x}_{6} \bar{x}_{16}, \\
  x_{1} x_{17} \bar{x}_{1} \bar{x}_{18}, & x_{16} x_{19} \bar{x}_{16} \bar{x}_{18}, 
  & x_{24} x_{19} \bar{x}_{24} \bar{x}_{20}, & x_{12} x_{21} \bar{x}_{12} \bar{x}_{20}, \\
  x_{4} x_{21} \bar{x}_{4} \bar{x}_{22}, & x_{1} x_{23} \bar{x}_{1} \bar{x}_{22}, 
  & x_{6} x_{23} \bar{x}_{6} \bar{x}_{24}, & x_{18} x_{24} \bar{x}_{18} \bar{x}_{1}.
\end{array}
\]}%
Here the generators are ordered along the orientation, starting from
the generator $x_1$ shown in Figure~\ref{figure:example-K_T}.
In the relators $\bar x$ denotes the inverse of~$x$.

Recall that the trefoil knot $T$ has the following Wirtinger presentation:
\[
G(T) 
= 
\langle
y_1, y_2 \mid y_1 y_2 y_1 = y_2 y_1 y_2 
\rangle .
\]

We define a map $f\colon G(K_T) \to G(T)$ as follows: 
{\allowdisplaybreaks\[
\begin{aligned}[t]
f(x_{1}) &= y_1y_2\bar y_1 y_2\bar y_1, \\ 
f(x_{2}) &= y_1\bar y_2 y_1\bar y_2 y_1\bar y_2 \bar y_1 y_2^3 \\&\qquad  \cdot \bar y_1 y_2\bar y_1 y_2\bar y_1, \\
f(x_{3}) &= y_1y_2\bar y_1 y_2\bar y_1^{\vphantom{1}}, \\
f(x_{4}) &= y_1\bar y_2 y_1y_2\bar y_1 y_2\bar y_1 y_2y_1\bar y_2 \bar y_1 y_2\bar y_1, \\
f(x_{5}) &= y_2y_1y_2\bar y_1 y_2\bar y_1 y_2y_1\bar y_2 \bar y_1 \bar y_2, \\
f(x_{6}) &= y_1\bar y_2 y_1y_2\bar y_1 y_2\bar y_1 y_2\bar y_1^{\vphantom{1}} , \\
f(x_{7}) &= y_1\bar y_2 y_1\bar y_2^3  y_1y_2^2\bar y_1 \bar y_1  y_2^3\bar y_1 y_2\bar y_1, \\
f(x_{8}) &= y_1\bar y_2 y_1y_2\bar y_1 y_2\bar y_1 y_2\bar y_1^{\vphantom{1}}, \\
f(x_{9}) &= y_1\bar y_2 y_1\bar y_2^2 y_1y_2\bar y_1 y_2\bar y_1  \\&\qquad  \cdot y_2^2\bar y_1 y_2\bar y_1, \\
f(x_{10}) &= y_1\bar y_2 y_1y_2\bar y_1 y_2\bar y_1 y_2\bar y_1^{\vphantom{1}} , \\
f(x_{11}) &= y_1y_2y_2\bar y_1^2, \\
f(x_{12}) &= y_1\bar y_2^2 y_1y_2^2\bar y_1^2 y_2^2 \bar y_1, \\
\end{aligned}
\quad
\begin{aligned}[t]
f(x_{13}) &= y_1y_2\bar y_1 y_2\bar y_1, \\
f(x_{14}) &= y_1\bar y_2 y_1\bar y_2^3 y_1y_2\bar y_1 y_2\bar y_1 y_2^3\bar y_1 y_2\bar y_1, \\
f(x_{15}) &= y_1y_2\bar y_1 y_2\bar y_1, \\
f(x_{16}) &= y_1\bar y_2^2 y_1y_2\bar y_1 y_2\bar y_1 y_2^2 \bar y_1, \\
f(x_{17}) &= y_1\bar y_2 y_1\bar y_2 y_1\bar y_2 \bar y_1 \bar y_2 y_1y_2\bar y_1  \\&\qquad  \cdot  y_2\bar y_1 y_2y_1y_2\bar y_1 y_2\bar y_1 y_2\bar y_1, \\
f(x_{18}) &= y_2^2\bar y_1, \\
f(x_{19}) &= y_1\bar y_2^2 y_1\bar y_2 y_1\bar y_2 \bar y_1 y_2^3 \bar y_1 y_2\bar y_1 y_2^2\bar y_1, \\
f(x_{20}) &= y_2^2\bar y_1, \\
f(x_{21}) &= y_1\bar y_2 y_1y_2\bar y_1 \bar y_2 y_1\bar y_2 y_1\bar y_2 \bar y_1 y_2\bar y_1^{\vphantom{1}} \\&\qquad  \cdot  y_2 y_1y_2\bar y_1 y_2\bar y_1 y_2y_1\bar y_2 \bar y_1 y_2\bar y_1^{\vphantom{1}}, \\
f(x_{22}) &= y_2^2\bar y_1, \\
f(x_{23}) &= y_1\bar y_2 y_1\bar y_2 \bar y_1 y_2^3 \bar y_1 y_2\bar y_1, \\
f(x_{24}) &= y_1\bar y_2 y_1y_2\bar y_1 \bar y_1 y_2^2 y_1\bar y_2 \bar y_1 y_2\bar y_1 .
\end{aligned}
\]}%

\begin{theorem}
  The map $f\colon G(K_T) \to G(T)$ is a non-meridional epimorphism. 
\end{theorem}

\begin{proof}
  It is shown that $f$ is a group homomorphism, by directly verifying that the
  relators of $G(K_T)$ vanish under~$f$.  For instance we have
  \begin{align*}
    f ( x_{6} x_{2} \bar{x}_{6} \bar{x}_{1} ) & = y_1 \bar y_2 y_1 y_1
    \bar y_1 y_1 \bar y_1 y_1 \bar y_1 \cdot y_1 \bar y_2 y_1 \bar y_2
    y_1 \bar y_2 \bar y_1 y_1^3 \bar y_1
    y_1 \bar y_1 y_1 \bar y_1 
    \\
    &\qquad \cdot y_1 \bar y_2 y_1 \bar y_2 y_1 \bar y_2 \bar y_1 y_1
    \bar y_1 \cdot y_1 \bar y_2 y_1 \bar y_2 \bar y_1 = e
    \\
    f( x_{10} x_{2} \bar{x}_{10} \bar{x}_{3} ) &= 
    y_1 \bar y_2 y_1 y_2 \bar y_1 y_2 \bar y_1 y_2 \bar y_1 \cdot 
    y_1 \bar y_2 y_1 \bar y_2 y_1 \bar y_2 \bar y_1 y_2^3 \bar
    y_1 y_2 \bar y_1 y_2 \bar y_1 
    \\
    &\qquad \cdot 
    y_1 \bar y_2 y_1 \bar y_2 y_1 \bar y_2 \bar y_1 y_2 \bar y_1 \cdot 
    y_1 \bar y_2 y_1 \bar y_2 \bar y_1 = e
  \end{align*}
  and so forth.

  To show that $f$ is an epimorphism, we explicitly describe elements
  of $G(K_T)$ which are sent to generators of~$G(T)$:
  \begin{align*}
    f( x_{18} x_{6} \bar{x}_{1} \bar{x}_{1} x_{18} x_{6} \bar{x}_{1} )
    &= y_2^2\bar y_1 \cdot y_1 \bar y_2 y_1 y_2 \bar y_1 y_2 \bar y_1 y_2 \bar y_1
    \cdot y_1 \bar y_2 y_1 \bar y_2 \bar y_1 \cdot y_1 \bar y_2 y_1 \bar y_2
    \bar y_1
    \\
    &\qquad
    \cdot y_2^2 \bar y_1 \cdot y_1 \bar y_2 y_1 y_2 \bar y_1 y_2 \bar y_1 
    \bar y_1 \cdot y_1 \bar y_2 y_1 \bar y_2 \bar y_1 \\
    &= y_2 y_1 y_2 \bar y_1 \bar y_2 y_1 \bar y_2 \bar y_1 y_2 y_1 y_2 \bar y_1 \bar y_1
    \\
    & = y_1 y_2 y_1 \bar y_1 \bar y_2 y_1 \bar y_2 \bar y_1 y_1 y_2 y_1 \bar y_1 \bar y_1
    = y_1 .
  \end{align*}
  Similarly we have the following:
  \[
  f(x_{1} \bar{x}_{6} \bar{x}_{18} x_{1} x_{18} x_{6} \bar{x}_{1}^2
  x_{18} x_{6} \bar{x}_{1}) = y_2 .
  \]

  Although we will show that there is no meridional epimorphism of
  $G(K_T)$ onto $G(T)$ in the next section, we present here a simple
  direct proof that our $f$ is not meridional.  Define a
  representation $\rho\colon G(T) \to \SL(2,\Z)$ by
  \[
  \rho(y_1) = \begin{bmatrix} 1 & 1 \\ 0 & 1 \end{bmatrix},\quad
  \rho(y_2) = \begin{bmatrix} 1 & 0 \\ -1 & 1\end{bmatrix}.
  \]
  It is straightforward to verify that $\rho$ is well-defined.  The
  image of $f(x_1)$ under $\rho$ is given by
  \[
  \rho ( f (x_1) ) = 
  \rho ( y_1 y_2 \bar y_1 y_2 \bar y_2 ) = 
  \begin{bmatrix}
    -1 & 2 \\ -3 & 5
  \end{bmatrix}.
  \]
  It has trace 4, while the trace of $\rho(y_1)$ is 2.  It follows
  that the image $f(x_1)$ of the meridian $x_1$ of $K_T$ is not
  conjugate to the meridian $y_1$ of~$T$.
\end{proof}

We remark that $K_T$ is a hyperbolic knot, according to
SnapPy~\cite{SnapPy}, and consequently, $K_T$ is prime.

\subsection{Construction using normal generators and Johnson's method}
\label{subsection:normal-generator-johnson-method}

To describe the second succesful example satisfying
Theorem~\ref{theorem:main-introduction}, we employ a more systematic
construction which combines algebraic computations in knot groups and
geometric realization arguments.  In the first step we construct a
non-meridional normal generator, and in the second step, we construct
knot group homomorphisms realizing the normal generator as the image
of a meridian.

\subsubsection*{Finding pseudo-meridian: twist knots}
\label{subsubsection:computation-for-twist-knots}

It is well known that a meridian is a normal generator of a knot
group.  We call a normal generator of a knot group a
\emph{pseudo-meridian}, that is, $w \in G(K)$ is called a
pseudo-meridian if $G(K)/ \langle w \rangle$ is trivial, where
$\langle w \rangle$ is the normal closure of $w$.

We will present useful pseudo-meridians of twist knots.  Let $J(2,2q)$
be the twist knot shown in Figure~\ref{figure:twist-knot} $(q \in
{\mathbb Z})$.  For example, $J(2,0)$ is the trivial knot, $J(2,2)$ is
the trefoil knot, and $J(2,-2)$ is the figure eight knot.  The
presentation of $G(J(2,2q))$ is given by
\[
 G(J(2,2q)) = 
\langle a,b ~ | ~ w^q a = b w^q \rangle, \quad w = [b,a^{-1}]. 
\]

\begin{figure}[ht]
  \labellist
  \small\hair 0mm
  \pinlabel {$2q$-crossings} at 107 40
  \endlabellist
  \includegraphics[scale=0.7]{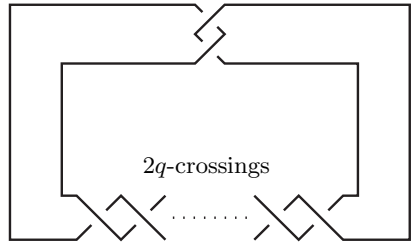}
  \caption{The twist knot $J(2,2 q)$.}
  \label{figure:twist-knot}
\end{figure}

\begin{proposition}\label{prop-pseudo-meridian-twistknot}
  Let 
  \[
  g_1 =\begin{cases}
    w^q a & \text{if $q > 0$},\\
    \bar{w}^q b & \text{if $g < 0$}.
  \end{cases}
  \]
  Then $g_1$ is a pseudo-meridian, but not a meridian for $q \neq 0$.
\end{proposition}

\begin{proof}
  First we show that $g_1$ is a pseudo-meridian by verifying that
  $G(J(2,2q))/\langle g_1 \rangle$ is trivial.  In the quotient, $g_1
  = e$, that is, $w^q = \bar{a}$ if $q > 0$.  Then the relation of the
  twist knot group gives $\bar{a} a = b \bar{a}$.  This implies that
  $b = a$ and $w = [b,\bar{a}] = e$.  Therefore $a = b = \bar{w}^q =
  e$.  Hence $g_1$ is a pseudo-meridian.  Similarly, we can show that
  $g_1 = \bar{w}^q b$ is a pseudo-meridian for $q < 0$.

  Next we show that $g_1$ is not conjugate to $a$, that is, $g_1$ is
  not a meridian of the twist knot.  Let $\rho : G(J(2,2q)) \to
  SL(2,{\mathbb C})$ be defined by
  \[
  \rho(a) = \begin{bmatrix} 1 & 1 \\ 0 & 1 \end{bmatrix},\quad
  \rho(b) = \begin{bmatrix} 1 & 0 \\ -u & 1\end{bmatrix},\quad  u \in {\mathbb C}-\{0\} .
  \]
  We set $\phi_q (u)$ by the $(1,1)$-entry of $\rho (w^q)$.  It is
  easy to see that if $u=0$, then $\rho$ is not a representation.
  More precisely, Riley~\cite{Riley:1972-1} showed that $\rho$ is a
  non-abelian parabolic representation if and only if $\phi_q (u) =
  0$.
  Moreover, Hoste-Shanahan~\cite{Hoste-Shanahan:2001-1} proved that
  $\phi_q (u)$ is irreducible and that $\deg \phi_q (u)$ is $2 q -1$
  if $q > 0$ and $2 |q|$ if $q < 0$.

  We define a polynomial $p_q (u)$ by 
  \[
  p_q (u) = 
  \begin{cases}
    \trace (\rho (w^q a)) - \trace (\rho (a)) = \trace (\rho (w^q a)) - 2
    & \text{if }q > 0, \\
    \trace (\rho (\bar{w}^q b)) - \trace (\rho (a)) = \trace (\rho (\bar{w}^q b)) - 2 
    & \text{if }q < 0. 
  \end{cases}
  \]
  First, consider the case $q > 0$.  Caley-Hamilton theorem gives us
  \[
  \rho (w)^2 - (\trace \rho (w)) \rho (w) + I = 0. 
  \]
  The trace of $\rho (w)$ is $u^2 + 2$.  Multiplying both sides by $\rho
  (w^{q-2} a)$, we obtain
  \[
  \rho (w^q a) = (u^2 + 2) \rho (w^{q-1} a) - \rho (w^{q-2} a). 
  \]
  Taking the trace of the both sides and using the definition of $p_q
  (u)$, we obtain a recursion formula for~$p_q (u)$:
  \[
  p_q (u) = (u^2 + 2) p_{q-1} (u) - p_{q-2} (u) + 2 u^2  
  \]
  (cf.\ \cite{Hoste-Shanahan:2001-1}).  Since $p_1 (u) = 2 u^2$ and
  $p_2 (u) = 2 u^4 + 6 u^2$, we conclude that $p_q (u)$ has a factor
  $u^2$ and $\deg p_q (u)$ is $2 q$.  Then $p_q (u)$ can be written as
  \[
  p_q (u) = u^2 \overline{p}_q (u)
  \]
  and the degree of $\overline{p}_q (u)$ is $2q-2$, which is less than
  $\deg \phi_q (u)$.  Since $\phi_q(u)$ is irreducible, $\phi_q (u)$
  does not have a multiple root, that is, $\phi_q (u)$ has distinct
  $2q-1$ roots, which are not zero.  Hence there exists at least one
  root $u \in {\mathbb C}$ of $\phi_q(u) = 0$ such that $p_q (u) \neq
  0$, namely,
  \[
  \trace (\rho (w^q a)) \neq \trace (\rho (a)). 
  \]
  This implies $w^q a$ is not conjugate with $a$.  Similarly, we prove
  the statement for the case $q < 0$.  The recursion formula of $p_q
  (u)$ is given by
  \[
  p_q (u) = (u^2 + 2) p_{q+1} (u) - p_{q+2} (u) + 2 u^2. 
  \]
  By straightforward computation, $p_{-1} (u) = 0$ and $p_{-2} (u) = 2
  u^2$.  Then the same argument as the case $q > 0$ holds.  Therefore
  $\bar{w}^q b$ is not conjugate with $a$.  This completes the proof.
\end{proof}

We can produce a generating set of the twist knot group by
conjugating~$g_1$.  In the case $q > 0$, the relation can be written
as $(g_1 \bar{a}) a = b (g_1 \bar{a})$.  Then $b= g a \bar{g}_1$ and
\[
w = [b,\bar{a}] = g_1 a \bar{g}_1 \cdot \bar{a} \cdot g_1 \bar{a}
\bar{g}_1 \cdot a = [g_1,a][g_1,\bar{a}] .
\]
Therefore we obtain 
\begin{align*}
  a 
  &= ([\bar{a},g_1] [a,g_1])^q g \\
  &= \bar{a} g_1 a \cdot \bar{g}_1 \cdot a g_1 \bar{a} \cdot \bar{g}_1 \cdots 
  \bar{a} g_1 a \cdot \bar{g}_1 \cdot a g_1 \bar{a} \\
  &= (g_2 \bar{g}_1 g_3 \bar{g}_1)^{q-1} g_2 \bar{g}_1 g_3 ,
\end{align*}
where $g_2 = \bar{a} g_1 a$ and $g_3 = a g_1 \bar{a}$.  It follows
that $G(J(2,2q))$ is generated by the three conjugate elements $g_1$,
$g_2$, and~$g_3$.

Similarly, in the case $q < 0$, we can show that $G(J(2,2q))$ is
generated by the three conjugate elements $g_1, g_2 = \bar{b} g_1 b,
g_3= b g_1 \bar{b}$.  Here $b$ can be expressed as $(g_3 \bar{g}_1
g_2 \bar{g}_1)^q g_1$.

\subsubsection*{Johnson's method for knot group epimorphs}
\label{subsubsection:johnson-method}

Gonz\'alez-Acu\~na and Johnson showed independently the following.

\begin{theorem}[\cite{Gonzalez-Acuna:1975-1,Johnson:1980-1}]
  \label{theorem:gonzales-acuna-johnson}
  Let $G$ be a group finitely generated by the conjugates of $g \in
  G$.  Then there is a knot with group $G(K)$ and meridian $\mu \in
  G(K)$, and a homomorphism $\varphi$ of $G(K)$ onto $G$ carrying
  $\mu$ onto $g$.
\end{theorem}

Johnson gave a proof of this result by presenting a process to
construct such a knot~$K$.  Here we review this process.

By hypothesis, $G$ is finitely generated and normally generated
by~$g$.  Thus we can choose finitely many generators of the form
$g_1=g$, $g_2= w_2 g \bar{w}_2$, \ldots, $g_n= w_n g \bar{w}_n$.  Take
a trivial link with $n$ components, and label the components by
$g_1,g_2,\ldots g_n$.  For each $i\ge 2$, we connect the component
$g_i$ to $g_1$ along a band which represent the word~$w_i$; for
example, if $w_2 = g_1 \bar{g}_3$, we connect the first circle and the
second circle as in Figure \ref{figure:connect-circles}.  Let $K$ be
the resulting knot.  Then it is not too difficult to see that $G(K)$
satisfies the conclusion of
Theorem~\ref{theorem:gonzales-acuna-johnson}.

\begin{figure}[H]
  \labellist
  \small\hair 0mm
  \pinlabel {$g_1$} at 36 107
  \pinlabel {$g_2$} at 158 107
  \pinlabel {$g_3$} at 285 107
  \endlabellist
  \vspace*{2em}
  \includegraphics{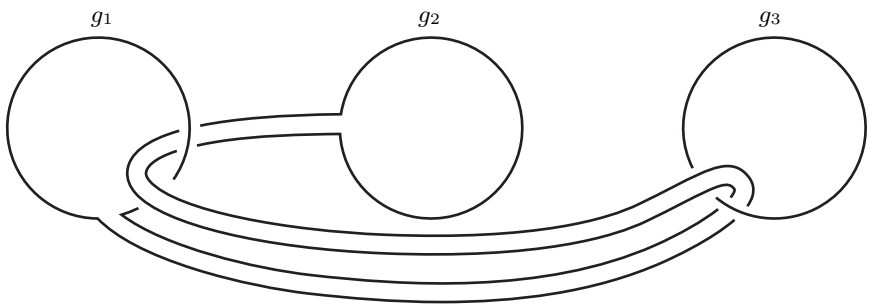}
  \caption{The connecting band for $w_2 = g_1 \bar{g}_3$.}
  \label{figure:connect-circles}
\end{figure}

For the twist knot group $G(J(2,2q))$, we use the pseudo-meridian
$g_1$ defined in Proposition~\ref{prop-pseudo-meridian-twistknot}.  By
the above process of Theorem \ref{theorem:gonzales-acuna-johnson}, we
can construct a knot $J_q$ for which there exists an epimorphism
$\varphi\colon G(J_q) \to G(J(2,2q))$ that sends a meridian of $J_q$
to the pseudo-meridian $g_1$ of $J(2,2q)$.  In particular, $\varphi$
is non-meridional.

\begin{figure}[H]
  \labellist
  \small\hair 0mm
  \pinlabel {$q$} at 11 90
  \pinlabel {$q$} at 384 289
  \endlabellist
  \includegraphics[scale=0.67]{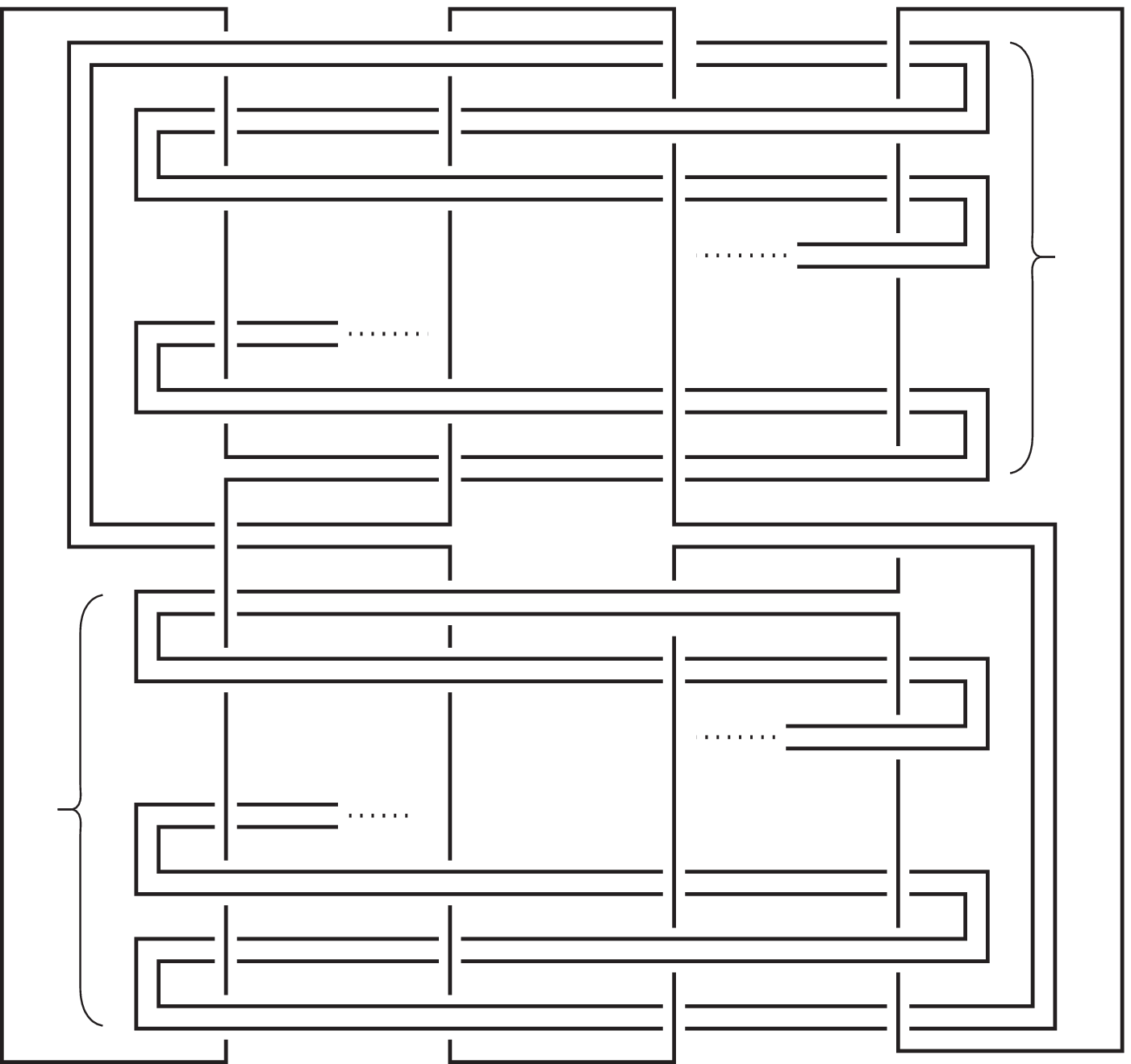}
  \caption{The knot $J_q$ $(q > 0)$.}
  \label{figure:Jplus}
\end{figure}

\begin{figure}[H]
  \labellist
  \small\hair 0mm
  \pinlabel {$|q|$} at 14 93
  \pinlabel {$|q|$} at 14 263
  \pinlabel {$|q|$} at 387 117
  \pinlabel {$|q|$} at 387 287     
  \endlabellist
  \includegraphics[scale=0.67]{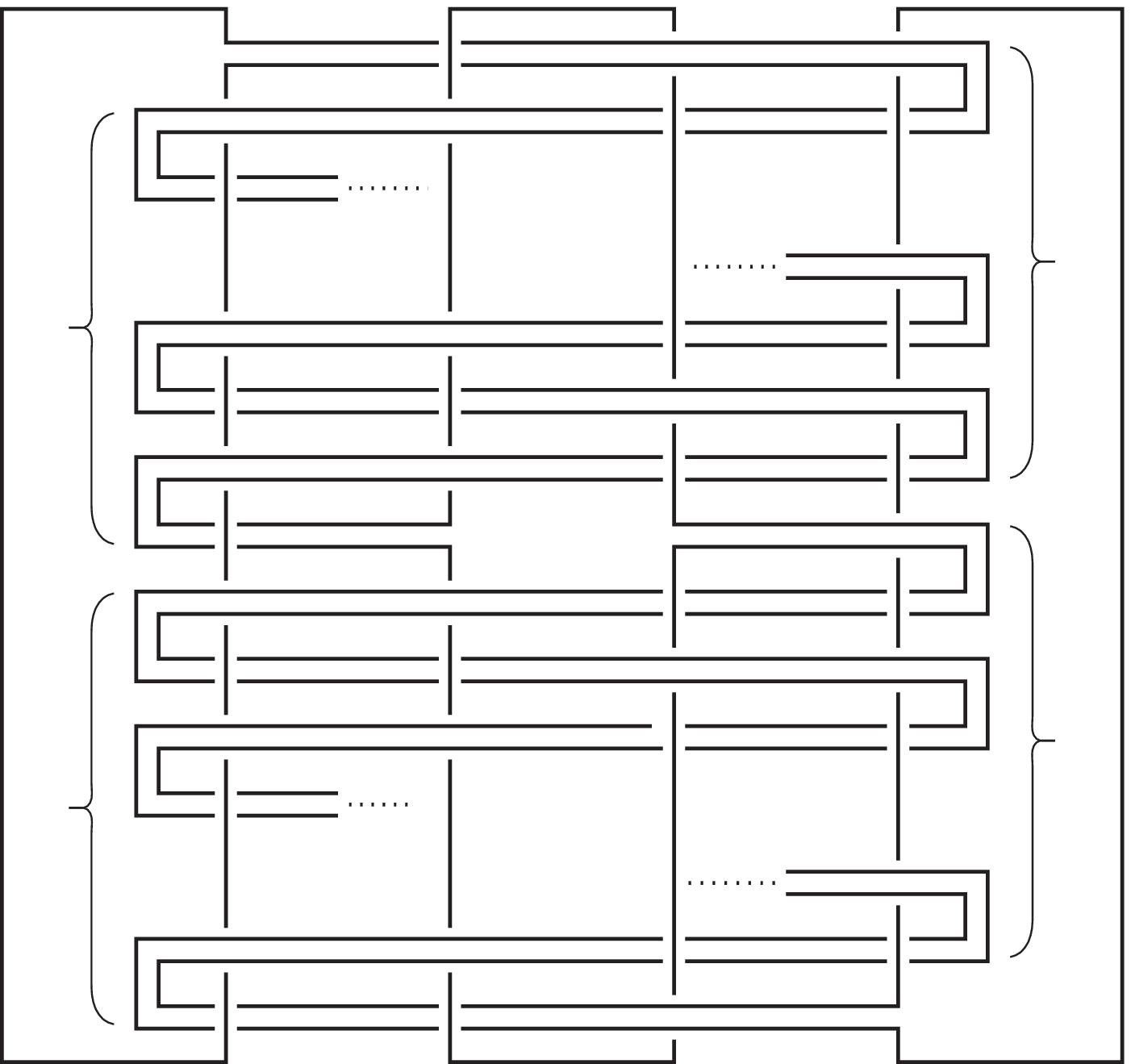}
  \caption{The knot $J_q$ $(q < 0)$.}
  \label{figure:Jminus}
\end{figure}

\section{Twisted Alexander polynomial and nonexistence of meridional epimorphisms}
\label{section:twisted-alexander-polynomial}

In this section we show that two of the examples described in
Section~\ref{section:construction-non-meridional-epi} do not admit any
meridional epimorphisms.  For this purpose we use twisted Alexander
polynomials, as discussed below.

\subsection{Obstructions to admitting a meridional epimorphism}
\label{subsection:obstruction-to-meridional-epi}

We recall the definition of the twisted Alexander polynomial following
Wada~\cite{Wada:1994-1}.  For this purpose, temporarily, we assume
that a knot is oriented.  For brevity, although any presentation of
a knot group can be used in general, we will consider only the special
case of a deficiency one presentation
\[
G=\langle x_1,\ldots,x_k \mid r_1,\ldots,r_{k-1}\rangle
\]
of a knot group $G=G(K)$ in which each generator $x_i$ represent a
positively oriented meridian.  For instance a Wirtinger presentation
can be used.  Suppose $\rho\colon G \to \SL(n,\mathbb F)$ is a
representation over a field $\mathbb F$.  Let $\alpha\colon G\to
\Z=\langle t\rangle$ be the surjection defined by $\alpha(x_i)=t$.
The tensor representation $\rho\otimes \alpha$ defined by
$(\rho\otimes\alpha)(g)=\alpha(g)\cdot \rho(g)$ gives rise to a ring
homomorphism $\Z[G] \to M_n(\mathbb F[t,t^{-1}])$.  Let $F\langle
x_i\rangle$ be the free group generated by the symbols
$x_1,\ldots,x_k$, and let $\Phi$ be the composition
\[
\Phi\colon\Z[F\langle x_i\rangle]\xrightarrow{\text{proj.}} \Z[G]
\xrightarrow{\rho\otimes\alpha} M_n(\mathbb F[t,t^{-1}]).
\]
Viewing $r_i$ as an element of $F\langle x_i \rangle$, the Fox
derivative $\frac{\partial r_i}{\partial x_j} \in \Z[F\langle
x_i\rangle]$ is defined as in~\cite{Fox:1953-1}.  The matrix
consisting of $(k-1)\times k$ blocks of size $n\times n$
\[
\Phi\Big(\frac{\partial r_i}{\partial x_j}\Big) \in M_n(\mathbb
F[t,t^{-1}]),\quad 1\le i\le k-1,\; 1\le j\le k
\]
is called a \emph{twisted Alexander matrix}.  It can be verified that
$\det\Phi(1-x_j)\ne 0$ for some~$j$ (see \cite[Lemma~2]{Wada:1994-1}).
For such an index $j$, let $M_j$ be the matrix obtained from the twisted
Alexander matrix by deleting the blocks on the $j$th column.  Viewing
$M_j$ as an $n(k-1)\times n(k-1)$ matrix over $\mathbb F[t,t^{-1}]$,
we define the \emph{twisted Alexander polynomial} for $(K,\rho)$ by
$\Delta_{K,\rho} = \Delta^N_{K,\rho}/\Delta^D_{K,\rho}$ where
\[
\Delta^N_{K,\rho} = \det(M_j),\quad \Delta^D_{K,\rho} = \det\Phi(x_j-1).
\]
We call $\Delta^N_{K,\rho}$ and $\Delta^D_{K,\rho}$ the
\emph{numerator} and \emph{denominator} of the twisted Alexander
polynomial.  Note that $\Delta^N_{K,\rho}$ and $\Delta^D_{K,\rho}$ are
Laurent polynomials in $\mathbb F[t,t^{-1}]$.

Under our assumption that the generators of the presentation for $G$
are positive meridians, both polynomials $\Delta^N_{K,\rho}$ and
$\Delta^D_{K,\rho}$ are well-defined invariants of $(K,\rho)$, up to
multiplication by units in $\mathbb F[t,t^{-1}]$.  It is a consequence
of the following two facts: (i)~the fraction $\Delta_{K,\rho}$ is
well-defined up to units for any choice of a presentation of $G$
(e.g.\ see \cite{Wada:1994-1}), and
(ii)~$\Delta^D_{K,\rho}=\det\Phi(x_j-1)$ is determined, up to units,
by the conjugacy class of the generator~$x_j$.

\begin{remark}
  By the same argument, the following more general statement is true:
  for a finitely presentable group $G$ and a conjugacy class $c$, the
  numerator and denominator of the twisted Alexander polynomial of $G$
  are invariants of $(G,c)$, provided that we use a presentation whose
  generators are in the conjugacy class~$c$.
\end{remark}

\begin{remark}
  For an unoriented knot $K$, $\Delta^N_{K,\rho}$ and
  $\Delta^D_{K,\rho}$ are well-defined up to the substitution $t\to
  t^{-1}$ (and up to multiplication by a unit).  For our purpose, it
  does not cause any problem.
\end{remark}

The following is the key ingredient we use to detect the non-existence
of a meridional epimorphism.

\begin{theorem}[Kitano-Suzuki-Wada~\cite{Kitano-Suzuki-Wada:2005-1,
    Kitano-Suzuki-Wada:2011-1}]
  \label{theorem:detecting-nonexistence-of-meridional-epi}
  Suppose there is a meridional epimorphism $G(K)\to G(K')$.  Then for
  any representation $\rho'\colon G(K') \to \SL(n,\mathbb F)$ over a
  field $\mathbb F$, there is a representation $\rho\colon G(K) \to
  \SL(n,\mathbb F)$ such that for some $\epsilon\in \{1,-1\}$,
  $\Delta^N_{K',\rho'}(t)$ divides $\Delta^N_{K,\rho}(t^\epsilon)$ in
  $\mathbb F[t,t^{-1}]$ and $\Delta^D_{K',\rho'}(t)$ is equal to
  $\Delta^D_{K,\rho}(t^\epsilon)$ up to units.
\end{theorem}

\begin{remark}\leavevmode\Nopagebreak
  \begin{enumerate}
  \item The representation $\rho$ in
    Theorem~\ref{theorem:detecting-nonexistence-of-meridional-epi} is
    the composition of the given meridional epimorphism and~$\rho'$.
    Since our aim is to detect the non-existence of a meridional
    epimorphism, we do not know what $\rho$ would be.  Consequently,
    to use
    Theorem~\ref{theorem:detecting-nonexistence-of-meridional-epi}.
    we need to investigate all representations $\rho\colon G(K)\to
    \SL(n,\mathbb{F})$ such that $\Im \rho = \Im \rho'$ and
    $\rho(\text{meridian of }K)$ is conjugate to $\rho(\text{meridian
      of }K')$ in~$\Im \rho'$.
  \item Under the weaker assumption that there is an epimorphism
    $G(K)\to G(K')$, which might be non-meridional, we have an weaker
    conclusion that $\Delta_{K',\rho'}(t)$
    divides~$\Delta_{K,\rho}(t^\epsilon)$.
    See~\cite{Kitano-Suzuki-Wada:2005-1, Kitano-Suzuki-Wada:2011-1}.
  \end{enumerate}
\end{remark}

\subsection{Computation for the first example on the trefoil}

In this subsection we show that there is no meridional epimorphism of
the knot group $G(K_T)$ onto the trefoil group $G(T)$, where
$K_T$ is the knot presented in Figure~\ref{figure:example-K_T}.

We will use representations over $\SL(2,\mathbb F_5)$.  In particular,
consider a representation
\[
\rho'\colon G(T) = \langle y_1,y_2 \mid y_1y_2y_1 = y_2y_1y_2
\rangle \to \SL(2,\mathbb F_5)
\]
defined by
\[
  \rho'(y_1) = \begin{bmatrix}0 & 4 \\ 1 & 3\end{bmatrix}, \quad
  \rho'(y_2) = \begin{bmatrix}0 & 1 \\ 4 & 3\end{bmatrix}.
\]
It is straightforward to verify that the relator is sent to the
identity.

A computer-aided computation shows that
\[
\Delta^N_{E,\rho'} = t^4 + 2 t^3 + 2 t^2 + 2 t + 1 , \quad 
\Delta^D_{E,\rho'} = t^2 + 2 t + 1.
\]
Note that both $\Delta^N_{T,\rho'}$ and $\Delta^D_{T,\rho'}$ are
symmetric.

In order to invoke
Theorem~\ref{theorem:detecting-nonexistence-of-meridional-epi}, we
compute the twisted Alexander polynomials of~$K_T$.  We are again
aided by computer programs written by the authors, which enumerates
all the representations $\rho\colon G(K_T) \to \SL(2,\mathbb F_5)$ (up
to conjugation) and then computes the associated twisted Alexander
polynomials.  By this we obtain that there are exactly eight distinct
twisted Alexander polynomials of $G(K_T)$ over $\SL(2,\mathbb F_5)$.
These polynomials are listed in
Appendix~\ref{subsection:polynomials-K_T-SL(2,F_5)},
Table~\ref{table:polynomials-of-K_T}.

From Table~\ref{table:polynomials-of-K_T}, it is straightforward to
verify that for any representation $\rho\colon G(K_T) \to
\SL(2,\mathbb F_5)$, $\Delta^D_{K_T,\rho} \ne \Delta^D_{T,\rho'}$ or
$\Delta^N_{K_T,\rho}$ does not divide $\Delta^N_{T,\rho'}$.  By
symmetry of the polynomials, the conclusion holds for
$(\Delta^N_{T,\rho'}(t^{-1}), \Delta^D_{T,\rho'}(t^{-1}))$ as well.
By Theorem~\ref{theorem:detecting-nonexistence-of-meridional-epi}, it
follows that there is no meridional epimorphism of $G(K_T)$
onto~$G(T)$.

\subsection{Computation for $J_{-1}$ and the figure eight knot}

In this subsection we show that there is no meridional epimorphism of
the knot group $G(J_{-1})$ onto the figure eight group $G(E)$, where
$J_{-1}$ is the knot described in
Section~\ref{subsection:normal-generator-johnson-method}.

For this case, we use representations over $\SL(2,\mathbb F_7)$.  Let
\[
\rho'\colon G(E) = \langle y_1,y_2 \mid 
\bar y_1 \bar y_2 y_1 y_2 \bar y_1 y_2 y_1 \bar y_2 \bar y_1 y_2
\rangle \to \SL(2,\mathbb F_7)
\]
be the representation defined by
\[
  \rho'(y_1) = \begin{bmatrix}0 & 4 \\ 5 & 2\end{bmatrix}, \quad
  \rho'(y_2) = \begin{bmatrix}1 & 0 \\ 3 & 1\end{bmatrix}.
\]
The relator of $G(E)$ is sent to the identity, and we have
\[
\Delta^N_{T,\rho'} = t^4 + t^3 + 3 t^2 + t + 1 , \quad 
\Delta^D_{T,\rho'} = t^2 + 5 t + 1.
\]

By computation aided by a computer, we obtain that there are exactly
139 representations up to conjugacy, and 58 distinct twisted Alexander
polynomials of $G(J_{-1})$ over $\SL(2,\mathbb F_7)$.  We list them in
Table~\ref{table:polynomials-of-J_-1} in
Appendix~\ref{subsection:polynomials-J_-1-SL(2,F_7)}.  From
Table~\ref{table:polynomials-of-J_-1} it is verified that for any
representation $\rho\colon G(J_{-1}) \to \SL(2,\mathbb F_7)$, either
$\Delta^D_{J_{-1},\rho} \ne \Delta^D_{E,\rho'}$ or
$\Delta^N_{J_{-1},\rho}$ does not divide~$\Delta^N_{E,\rho'}$.  Since
both $\Delta^D_{E,\rho'}(t)$ and $\Delta^N_{E,\rho'}(t)$ are
symmetric, the conclusion holds for $(\Delta^N_{E,\rho'}(t^{-1}),
\Delta^D_{E,\rho'}(t^{-1}))$ as well.  By
Theorem~\ref{theorem:detecting-nonexistence-of-meridional-epi}, it
follows that there is no meridional epimorphism of $G(J_{-1})$
onto~$G(E)$.

\begin{remark}
  The most time-consuming step of the computation is to find all the
  representations of the given knot group.  Our implementation
  performs a brute-force search; since its running time is exponential
  to the number of the generators, it would be intereseting if a more
  clever algorithm is available.  For $G(K_T)$, there is a
  presentation with 3 generators, and all the 37 $\SL(2,\mathbb{F}_5)$
  representations (up to conjugacy) are found within a few seconds.
  For $G(J_{-1})$, we use a simplified presentation with 5 generators,
  and all the 139 $\SL(2,\mathbb{F}_7)$ representations (up to
  conjugacy) are found within 2 minutes.  For $G(J_q)$ for $q\ge 1$ or
  $q<-1$, we could not derive any conclusion within resonable running
  time; computation for two weeks on a computer with a 3GHz Intel i7
  processor was not enough.
\end{remark}

\section{Satellite construction and knot group epimorphisms}
\label{section:satellite-construction}

In this section we present a method to obtain infinitely many pairs of
knots $(K,K')$ for which $G(K')$ is a non-meridional homomorph of
$G(K)$ but there is no meridional epimorphism $G(K)\to G(K')$.  We
will start with a given ``seed'' pair $(K,K')$ of knots with the
desired property, and then apply certain satellite constructions to
produce infinitely many such examples.

\subsection{Satellite construction and knot group epimorphisms}

We begin by recalling the standard satellite construction.  Let $K$ be
a knot in $S^3$, and $\alpha$ is an unknotted oriented embedded circle
in $S^3$ disjoint to~$K$.  Glue the exterior $E_\alpha$ of
$\alpha\subset S^3$ and the exterior $E_J$ of another knot $J$ in
$S^3$ along an orientation reversing diffeomorphism between the
boundary tori, which identifies a zero-linking longitude and
positively oriented meridian of $\alpha$ with a positively oriented
meridian and zero-linking longitude of~$J$, respectively.  There is a
diffeomorphism of the resulting 3-manifold onto $S^3$, and the image
of $K$ under the diffeomorphism is a new knot, which we denote by
$K(\alpha,J)$.  In traditional terminology, $J$ is the
\emph{companion} and $K$ viewed as a knot in the solid torus
$E_\alpha$ is the \emph{pattern}.  As illustrated in
Figure~\ref{figure:satellite-example}, $K(\alpha,J)$ is the knot
obtained by ``tying'' $J$ into $K$ along a 2-disk bounded by~$\alpha$.

\begin{figure}[H]
  \labellist
  \pinlabel {\LARGE$R$} at 27 60
  \pinlabel {\LARGE$R$} at 212 60
  \pinlabel {$\alpha$} at 120 72
  \pinlabel {$J$} at 280 59
  \endlabellist
  \includegraphics{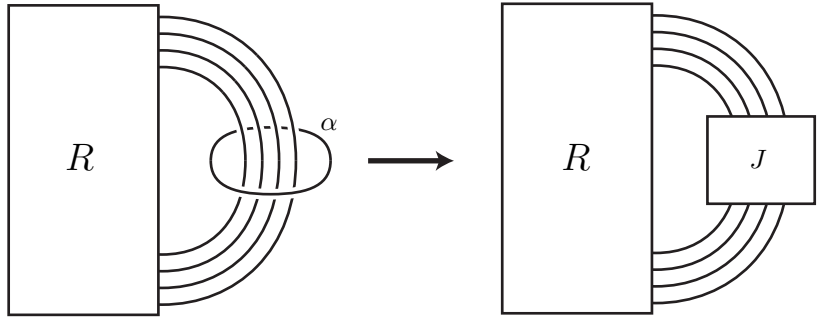}
  \caption{A satellite construction.}
  \label{figure:satellite-example}
\end{figure}

Note that $E_{K\sqcup\alpha}$ is a subspace of $E_{K(\alpha,J)}$ and a
subspace of~$E_K$.  Let $i_*$ and $j_*$ be the inclusion-induced
homomorphisms of $G(K\sqcup\alpha)$ into $G(K(\alpha,J))$ and $G(K)$,
respectively.  Then we have the following folklore:

\begin{lemma}
  \label{lemma:satellite-epimorphism}
  There exists a meridional epimorphism $q\colon G(K(\alpha,J)) \to G(K)$
  satisfying $i_*= q\circ j_*$.
  \[
  \begin{diagram}
    \node[2]{G(K(\alpha,J))}\arrow[2]{s,r}{q}
    \\
    \node{G(K\cup\alpha)}\arrow{ne,t}{j_*}\arrow{se,b}{i_*}
    \\
    \node[2]{G(K)}
  \end{diagram}
  \]
\end{lemma}

For the reader's convenience, we give a proof.

\begin{proof}
  There is a degree one map $E_J$ to the trivial knot exterior
  $S^1\times D^2$ which extends a diffeomorphism on the boundary
  sending the meridian and zero-linking longitude of $J$ to those of
  the unknot.  Glueing it with the identity map on $E_{K\cup\alpha}$,
  one obtains a degree one map $E_{K(\alpha,J)} \to E_K$ which induces
  the desired~$q$.
\end{proof}

We recall some standard definitions.  The \emph{derived subgroups} of
a group $G$ is defined inductively by $G^{(0)}=G$,
$G^{(n+1)}=[G^{(n)}, G^{(n)}]$ where $[A,B]$ denotes the subgroup
generated by commutators $\{aba^{-1}b^{-1}\mid a\in A,\,b\in B\}$.
For the first infinite ordinal $\omega$, the transfinite derived
subgroup $G^{(\omega)}$ is defined by $G^{(\omega)}
=\bigcap_{n<\infty} G^{(n)}$.  A group $G$ is \emph{residually
  solvable} if $G^{(\omega)} = \{e\}$.

\begin{remark}
  \leavevmode\Nopagebreak
  \begin{enumerate}
  \item The following well-known examples will be useful for our
    purpose: if $K$ is a fibered knot (or link), then $G(K)$ is
    residually solvable.  For, the commutator subgroup
    $G(K)^{(1)}=[G(K),G(K)]$ is the fundamental group of a surface
    with nonempty boundary, and thus a free group.  It is known that a
    free group is residually solvable.
  \item Not all knot groups are residually solvable.  For example, the
    group of a nontrivial knot $K$ with Alexander polynomial one is
    not residually solvable.
  \end{enumerate}
\end{remark}

The key technical ingredient we use is the following.

\begin{theorem}
  \label{theorem:satellite-meridional-epimorphisms}
  Suppose $T$ is a knot with residually solvable group~$G(T)$.
  Suppose $J$ has Alexander polynomial one.  Then for any homomorphism
  $f\colon G(K(\alpha,J))\to G(T)$, there is an induced homomorphism
  $f'\colon G(K)\to G(T)$ that makes the following diagram commute:
  \[
  \begin{diagram}
    \node[2]{G(K(\alpha,J))}\arrow{se,t}{f}\arrow[2]{s,r}{q}
    \\
    \node{G(K\cup\alpha)}\arrow{ne,t}{j_*}\arrow{se,b}{i_*}
    \node[2]{G(T)}
    \\
    \node[2]{G(K)}\arrow{ne,b}{f'}
  \end{diagram}
  \]
  In addition, the following hold:
  \begin{enumerate}
  \item if $f$ is an epimorphism, then $f'$ is an epimorphism;
  \item if $f$ is meridional, then $f'$ is meridional.
  \end{enumerate}
\end{theorem}

The following corollary is an immediate consequence.

\begin{corollary}
  \label{corollary:satellite-non-meritional-epimorphism}
  If $T$ and $J$ are as in
  Theorem~\ref{theorem:satellite-meridional-epimorphisms} and there is
  no meridional epimorphism of $G(K)$ onto $G(T)$, then there is no
  meridional epimorphism of $G(K(\alpha,J))$ onto~$G(T)$.
\end{corollary}

\begin{proof}[Proof of Theorem~\ref{theorem:satellite-meridional-epimorphisms}]
  Consider the composition
  \[
  g\colon G(J) \to G(K(\alpha,J))\to G(T).
  \]
  Since $J$ has Alexander polynomial one, the commutator subgroup
  $[G(J),G(J)]$ is perfect, i.e.,
  $G(J)^{(1)}=G(J)^{(2)}=\cdots=G(J)^{(\omega)}$.  Since
  $g(G(J)^{(\omega)})\subset G(T)^{(\omega)}$ and $G(T)$ is residually
  solvable, it follows that $G(J)^{(1)} \subset \Ker g$.  Therefore
  $g$ factors through $G(J)/[G(J),G(J)]=\Z$.

  From the construction of $K(\alpha,J)$, one sees that
  $G(K(\alpha,J))$ is the amalgamated product of $G(K\cup \alpha)$ and
  $G(J)$ over $\pi_1(S^1\times S^1)=\Z^2$.  The homomorphism
  \[
  G(K(\alpha,J)) = G(K\cup\alpha)\amalg_{\Z^2} G(J) \to
  G(K\cup\alpha)\amalg_{\Z^2} \Z = G(K)
  \]
  induced by the abelianization $G(J) \to \Z$ is our $q$ in
  Lemma~\ref{lemma:satellite-epimorphism}.  From the observation in
  the previous paragraph, it follows that $f\colon G(K(\alpha,J))\to
  G(T)$ induces a homomorphism of~$G(K)$.  This proves the first
  conclusion.

  From the commutative diagram the remaining conclusions follow
  immediately.
\end{proof}

\subsection{Infinitely many examples}

\begin{theorem}
  \label{theorem:infinitely-many-examples}
  Let $T$ be the trefoil knot or the figure eight knot.  Then there
  are infinitely many prime knots $K_1,K_2,\ldots$ satisfying the
  following:
  \begin{enumerate}
  \item For each $n$ there is an epimorphism $G(K_n) \to G(T)$ but
    there is no meridional epimorphism $G(K_n)\to G(T)$.
  \item $K_n$ and $K_m$ are not equivalent for any $n\ne m$.
  \end{enumerate}
\end{theorem}

\begin{remark}
  Due to Silver and Whitten~\cite{Silver-Whitten:2006-1}, for prime
  knots there exists a meridional epimorphism if and only if there
  exists an epimorphism preserving peripheral subgroups.  Therefore
  for our $K_n$ in Theorem~\ref{theorem:infinitely-many-examples},
  there is no epimorphism $G(K_n)\to G(T)$ sending peripheral
  subgroups into peripheral subgroups.
\end{remark}

\begin{proof}
  Recall that we have constructed a prime knot $K$ which admits an
  epimorphism $G(K)\to G(T)$ but does not admit any meridional
  epimorphism $G(K)\to G(T)$; when $T$ is the trefoil, $K = K_T$ in
  Figure~\ref{figure:example-K_T}, and when $T$ is the figure eight,
  $K=J_{-1}$ in Figure~\ref{figure:Jminus}.

  Choose an embedded circle $\alpha$ in $S^3-K$ satisfying the
  following: $\alpha$ is unknotted in $S^3$, $\alpha$ does not bound a
  2-disk in~$E_K$, and $K\cup \alpha$ is a prime link.  For example,
  we may use $\alpha$ shown in
  Figures~\ref{figure:satellite-curve-K_T}
  and~\ref{figure:satellite-curve-J_-1}.  Choose a hyperbolic
  nontrivial knot $P$ with Alexander polynomial one; for example the
  knot \verb-11n_34- in KnotInfo~\cite{Cha-Livingston:KnotInfo} can be
  used as~$P$.  Let $P_n$ be the connected sum of $n$ copies of~$P$.
  Let $K_n=K(\alpha,P_n)$.

  \begin{figure}[H]
    \labellist\small
    \pinlabel {$K_T$} at 30 20
    \pinlabel {$\alpha$} at 10 45
    \endlabellist
    \includegraphics{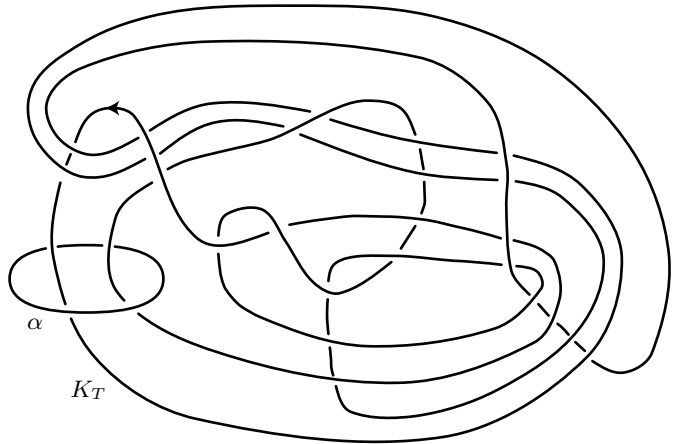}
    \label{figure:satellite-curve-K_T}
    \caption{A curve $\alpha$ for a satellite construction on~$K_T$.
      The link $K_T\cup \alpha$ is hyperbolic with volume $23.2123$,
      according to SnapPy~\cite{SnapPy}.}
  \end{figure}

  \begin{figure}[H]
    \labellist\small
    \pinlabel {$J_{-1}$} at 150 62
    \pinlabel {$\alpha$} at 150 35
    \pinlabel {$x_1$} at 100 146
    \endlabellist
    \includegraphics{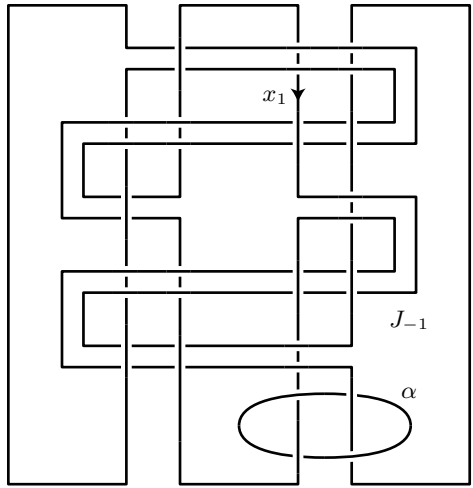}
    \label{figure:satellite-curve-J_-1}
    \caption{A curve $\alpha$ for a satellite construction
      on~$J_{-1}$.  The link $J_{-1}\cup \alpha$ is hyperbolic with
      volume $26.2914$, according to SnapPy~\cite{SnapPy}.}
  \end{figure}
  
  Note that $G(T)$ is residually solvable since $T$ is fibered.  Since
  the seed knot $K$ satisfies the conclusion (1), our $K_n$ satisfies
  (1) by Lemma~\ref{lemma:satellite-epimorphism} and
  Theorem~\ref{theorem:satellite-meridional-epimorphisms}.

  To show that the $K_n$ have distinct knot types, we consider the JSJ
  decomposition of the exterior.  Indeed in most cases satellite
  constructions with distinct companion give distinct JSJ
  decompositions.  In our case it can be seen as follows.  The
  exterior of $E_{K_n}$ is the union of $E_{K\cup \alpha}$
  and~$E_{P_n}$.  Since $P_n$ is a nontrivial knot and $\alpha$ does
  not bound a 2-disk in $E_K$, the boundary of $E_{P_n}$ is an
  incompressible torus in~$E_{K_n}$.  Since $P$ is hyperbolic, the JSJ
  tori of $E_{K_n}$ are exactly the union of $\partial E_{P_n}$ and
  the JSJ tori of $E_{K\cup\alpha}$ and~$E_{P_n}$.  In addition, since
  $P_n$ is the connected sum of $n$ copies of the hyperbolic knot $P$,
  the number of the JSJ tori of $E_{K_n}$ is monotonically increasing.
  It follows that $K_n$ and $K_m$ have non-homeomorphic exteriors if
  $n\ne m$.

  For the primality of the $K_n$, one may again look at the JSJ
  decomposition: since $K\cup \alpha$ is prime, the root piece in the
  JSJ decomposition of $E_{K_n}$ is not a composition space and
  therefore $K_n$ is not a composite knot.

  Or alternatively, one can give a direct argument as follows: suppose
  there is a 2-sphere $S$ in $S^3$ meeting $K_n$ at two points.  Let
  $\Sigma$ be the boundary of $E_{J_n} \subset E_{K_n}$ and look at
  the intersection of $S$ and $\Sigma$.  An innermost circle $C$ of
  $S\cap \Sigma$ must bound a 2-disk on $\Sigma$, since otherwise it
  would be a compressing disk of the incompressible torus~$\Sigma$.
  So we can remove the intersection circle $C$ by isotopying $S$
  in~$S^3$.  Repeating this we may assume that $S$ is disjoint
  to~$\Sigma$.  Now, $S$ can be viewed as a 2-sphere in $S^3$ which
  meets $K$ at two points and disjoint to~$\alpha$.  Let $B_1$ and
  $B_2$ be the 3-balls obtained by cutting $S^3$ along~$S$ where
  $\alpha\subset B_1$.  Since $K\cup \alpha$ is prime, $B_2\cap K$ is
  an unknotted arc in $B_2$.  This shows that $S$ does not give a
  nontrivial connected sum decomposition of~$K_n$.
\end{proof}

\appendix

\section{Tables of twisted Alexander polynomials}
\label{section:tables-of-polynomials}

\subsection{Twisted Alexander polynomials of $K_T$ over $\SL(2,\mathbb F_5)$}
\label{subsection:polynomials-K_T-SL(2,F_5)}

In this subsection we present all the twisted Alexander polynomials,
associated to representations over $\SL(2,\mathbb F_5)$, of the knot
$K_T$ shown in Figure~\ref{figure:example-K_T}.

We use a simpler presentation of $G(K_T)$ to find out all the
representations.  A straightforward simplification of the presentation
of $G(K_T)$ used in Section~\ref{subsection:first-example} shows that
$x_1$,$x_6$, and $x_{17}$ generate~$G(K_T)$.  In fact these three
generators and the following relators form a presentation of~$G(K_T)$:
\begin{gather*}
x_{1} x_{17} \bar x_{1} x_{6} (x_{17} \bar x_{6} x_{1} \bar x_{17}
\bar x_{1})^3 x_{17} \bar x_{1} \bar x_{6} (x_{1} \bar x_{17} x_{1}
x_{17} \bar x_{1} x_{6})^2 (\bar x_{17} x_{1} x_{17} \bar x_{1}
x_{6})^2 \bar x_{17} (\bar x_{6} x_{1} \bar x_{17} \bar x_{1})^2
x_{6},
\\
 x_{1} (\bar x_{17} x_{1} x_{17} \bar x_{1} x_{6})^3 \bar x_{17}
(\bar x_{6} x_{1} \bar x_{17} \bar x_{1} x_{17})^4 \bar x_{1} \bar
x_{6} x_{1} (\bar x_{17} x_{1} x_{17} \bar x_{1} x_{6})^4 (x_{17} \bar
x_{6} x_{1} \bar x_{17} \bar x_{1})^3 x_{17}.
\end{gather*}
We use computer programs to find out all the representations over
$\SL(2,\mathbb F_5)$ and to compute the associated twisted Alexander
polynomials.  It turns out that there are exactly 19 representations
modulo conjugacy, which give 8 distinct twisted Alexander polynomials.
We list the 8 polynomials in Table~\ref{table:polynomials-of-K_T}.
For each twisted Alexander polynomial in
Table~\ref{table:polynomials-of-K_T}, we describe a representation
$\rho\colon G(K_T)\to \SL(2,\mathbb F_5)$ (which is not unique in
general) which gives the polynomial by presenting the matrices
$\rho(x_1)$, $\rho(x_6)$, and~$\rho(x_{17})$.

\begin{table}[H]
  \caption{The twisted Alexander polynomials of the knot $K_T$ 
    in Figure~\ref{figure:example-K_T} over $\SL(2,\mathbb F_5)$.}
  \label{table:polynomials-of-K_T}
  \[
  \newdimen\tw \newbox\tb \setbox\tb\hbox{$\rho(x_{16})$ } \tw=\wd\tb
  \begin{array}{>{\centering$}p{\tw}<{$}@{\ }>{\centering$}p{\tw}<{$}@{\ }>{\centering$}p{\tw}<{$}cc}
    \toprule
    \rho(x_1) & \rho(x_6) & \rho(x_{17}) & \Delta^N_{K_T,\rho} & \Delta^D_{K_T,\rho} \\
    \midrule
    \sbmatrix{1&0\\0&1} & \sbmatrix{1&0\\0&1} & \sbmatrix{1&0\\0&1} & 
    t^8 + t^7 + 4 t^5 + 4 t^4 + 4 t^3 + t + 1 & t^2 + 3 t + 1 \\[.6ex]
    \sbmatrix{2&4\\4&1} & \sbmatrix{2&4\\4&1} & \sbmatrix{2&4\\4&1} & 
    t^8 + 4 t^7 + t^5 + 4 t^4 + t^3 + 4 t + 1 & t^2 + 2 t + 1 \\[.6ex]
    \sbmatrix{0&2\\2&4} & \sbmatrix{0&2\\2&4} & \sbmatrix{0&2\\2&4} & 
    t^8 + 2 t^7 + t^6 + 2 t^5 + 4 t^4 + 2 t^3 + t^2 + 2 t + 1 & t^2 + t + 1 \\[.6ex]
    \sbmatrix{0&2\\2&4} & \sbmatrix{0&4\\1&4} & \sbmatrix{1&3\\4&3} & 
    t^8 + 2 t^7 + t^5 + 3 t^4 + t^3 + 2 t + 1 & t^2 + t + 1 \\[.6ex]
    \sbmatrix{1&1\\4&0} & \sbmatrix{1&1\\4&0} & \sbmatrix{1&1\\4&0} & 
    t^8 + 3 t^7 + t^6 + 3 t^5 + 4 t^4 + 3 t^3 + t^2 + 3 t + 1 & t^2 + 4 t + 1 \\[.6ex]
    \sbmatrix{1&1\\4&0} & \sbmatrix{2&2\\1&4} & \sbmatrix{3&4\\2&3} & 
    t^8 + 3 t^7 + 4 t^5 + 3 t^4 + 4 t^3 + 3 t + 1 & t^2 + 4 t + 1 \\[.6ex]
    \sbmatrix{3&2\\0&2} & \sbmatrix{3&2\\0&2} & \sbmatrix{3&2\\0&2} & 
    t^8 + 3 t^6 + 3 t^4 + 3 t^2 + 1 & t^2 + 1 \\[.6ex]
    \sbmatrix{3&2\\0&2} & \sbmatrix{0&3\\3&0} & \sbmatrix{0&3\\3&0} & 
    t^8 + 4 t^6 + t^4 + 4 t^2 + 1 & t^2 + 1 \\[.2ex]
    \bottomrule
  \end{array}
  \]
\end{table}

\subsection{Twisted Alexander polynomials of $J_{-1}$ over $\SL(2,\mathbb F_7)$}
\label{subsection:polynomials-J_-1-SL(2,F_7)}

In this subsection we present all the twisted Alexander polynomials,
over $\SL(2,\mathbb F_7)$, of the knot $J_{-1}$ given in
Figure~\ref{figure:Jminus}.  We start with the standard Wirtinger
representation of~$J_{-1}$; the generators are labeled by $x_1, x_2,
\ldots, x_{32}$, starting from the arc with an arrow in
Figure~\ref{figure:Jminus}.  By simplifying the Wirtinger
presentation, we obtain a deficiency one presentation of $G(J_{-1})$
whose generators are the Wirtinger generators $x_2, x_{10}, x_{19},
x_{22}, x_{30}$, and relators are the following 4 words:
\begin{gather*}
  x_{19} \bar x_{30} \bar x_{10} x_{30} \bar x_{19} x_{10} \bar x_{2}
  x_{30} x_{2} \bar x_{30} x_{2} \bar x_{10} x_{19} \bar x_{30} x_{10}
  x_{30} \bar x_{19} x_{10} \bar x_{2} \bar x_{30} x_{2} \bar x_{10},
  \\
  x_{30} x_{19} \bar x_{30} \bar x_{10} x_{30} \bar x_{19} x_{22}
  (x_{30} \bar x_{19})^2 (\bar x_{30} x_{19})^2 \bar x_{30} \bar
  x_{22} x_{19} \bar x_{30} x_{10} x_{30} \bar x_{19} \bar x_{30}
  x_{2} \bar x_{10} x_{22} x_{10} \bar x_{2},
  \\
  \begin{split}
    \bar x_{30} x_{2} \bar x_{10} x_{19} \bar x_{30} x_{10} x_{30}
    \bar x_{19} \bar x_{30} x_{2} \bar x_{10} x_{22} x_{10} x_{2} \bar
    x_{10} \bar x_{22}
    \\
    \hbox{}\cdot (x_{10} \bar x_{2} x_{30} x_{19} \bar x_{30} \bar
    x_{10} x_{30} \bar x_{19})^2 x_{10} \bar x_{2} x_{30} \bar x_{10},
  \end{split}
  \\
  \begin{split}
    x_{19} \bar x_{30} \bar x_{22} x_{19} \bar x_{30} x_{10} x_{2}
    \bar x_{10} \bar x_{22} x_{10} \bar x_{2} x_{30} x_{19} \bar
    x_{30} \bar x_{10} (x_{30} \bar x_{19} x_{22})^2 x_{19} \bar
    x_{30} \bar x_{22} x_{19} \bar x_{30} x_{10} x_{30}
    \\
    \qquad \hbox{}\cdot \bar x_{19} \bar x_{30} x_{2} \bar x_{10}
    x_{22} x_{10} \bar x_{2} \bar x_{10} x_{30} \bar x_{19} x_{22}
    (x_{30} \bar x_{19})^2 \bar x_{30}.
  \end{split}
\end{gather*}
Again, we use computer programs which enumerate all the
representations of this presentation over $\SL(2,\mathbb F_7)$ and
compute the associated twisted Alexander polynomials.  There are total
139 representations up to conjugation, from which 58 distinct twisted
Alexander polynomials are obtained.  The polynomials are listed in
Table~\ref{table:polynomials-of-J_-1}.  For each polynomial in the
tables, we also describe a representation $\rho$ (which is not unique
in general) which gives the polynomial.

{\bigskip
\def\tw{\dimen255} \newbox\tb
\setbox\tb\hbox{$\rho(x_{16})$} \tw=\wd\tb
\def\tp{\vphantom{1^{1^{1^1}}}}
\def\bp{\vphantom{1_{1_{1_1}}}}
\begin{longtable}{>{\centering$}p{\tw}<{$}@{\ }>{\centering$}p{\tw}<{$}@{\ }>{\centering$}p{\tw}<{$}@{\ }>{\centering$}p{\tw}<{$}@{\ }>{\centering$}p{\tw}<{$}>{$}c<{$}>{$}c<{$}}
  \caption{Twisted Alexander polynomials of the knot $J_{-1}$.}
  \label{table:polynomials-of-J_-1}
  \\
  \toprule
  \rho(x_2) & \rho(x_{10}) & \rho(x_{19}) & \rho(x_{22}) & \rho(x_{30}) & \Delta^N_{J_{-1},\rho} & \Delta^D_{J_{-1},\rho} \\
  \midrule
  \endfirsthead
  \toprule
  \rho(x_2) & \rho(x_{10}) & \rho(x_{19}) & \rho(x_{22}) & \rho(x_{30}) & \Delta^N_{J_{-1},\rho} & \Delta^D_{J_{-1},\rho} \\
  \midrule
  \multicolumn{7}{l}{(continued from previous page)}\\[.9ex]
  \endhead
  \multicolumn{7}{r}{(continued on next page)}\\
  \bottomrule
  \endfoot
  \bottomrule
  \endlastfoot
  \sbmatrix{0&4\\5&2} & \sbmatrix{0&4\\5&2} & \sbmatrix{1&0\\3&1} & \sbmatrix{1&2\\0&1} & \sbmatrix{2&2\\3&0} & \begin{aligned} t^8 + t^7 + 3 t^6 + 5 t^5 + t^4 \\[-.8ex] + 5 t^3 + 3 t^2 + t + 1  \bp\end{aligned} & t^2 + 5 t + 1 \\[0ex]
  \sbmatrix{5&0\\0&3} & \sbmatrix{0&6\\1&1} & \sbmatrix{5&0\\4&3} & \sbmatrix{3&0\\4&5} & \sbmatrix{6&4\\1&2} & \begin{aligned}\tp t^8 + 2 t^6 + 3 t^4 + 2 t^2 + 1 \bp\end{aligned} & t^2 + 6 t + 1 \\[0ex]
  \sbmatrix{0&6\\1&0} & \sbmatrix{0&6\\1&0} & \sbmatrix{0&6\\1&0} & \sbmatrix{0&6\\1&0} & \sbmatrix{0&6\\1&0} & \begin{aligned}\tp t^8 + 2 t^4 + 1 \bp\end{aligned} & t^2 + 1 \\[0ex]
  \sbmatrix{4&0\\0&2} & \sbmatrix{4&3\\0&2} & \sbmatrix{2&0\\5&4} & \sbmatrix{0&1\\6&6} & \sbmatrix{2&0\\1&4} & \begin{aligned}\tp t^8 + 5 t^7 + t^6 + 5 t^5 + 4 t^4 \\[-.8ex] + 5 t^3 + t^2 + 5 t + 1 \bp\end{aligned} & t^2 + t + 1 \\[0ex]
  \sbmatrix{0&6\\1&4} & \sbmatrix{1&3\\3&3} & \sbmatrix{3&1\\2&1} & \sbmatrix{0&2\\3&4} & \sbmatrix{6&3\\5&5} & \begin{aligned}\tp 2 t^8 + t^7 + 4 t^6 + t^4 + 4 t^2 \\[-.8ex] + t + 2 \bp\end{aligned} & t^2 + 3 t + 1 \\[0ex]
  \sbmatrix{0&4\\5&5} & \sbmatrix{4&2\\5&1} & \sbmatrix{6&2\\0&6} & \sbmatrix{2&4\\3&3} & \sbmatrix{6&4\\0&6} & \begin{aligned}\tp 4 t^4 + 2 t^3 + 3 t^2 + 2 t + 4 \bp\end{aligned} & t^2 + 2 t + 1 \\[0ex]
  \sbmatrix{0&4\\5&5} & \sbmatrix{6&2\\0&6} & \sbmatrix{0&2\\3&5} & \sbmatrix{2&2\\6&3} & \sbmatrix{4&1\\3&1} & \begin{aligned}\tp 4 t^8 + 5 t^7 + 3 t^6 + 3 t^5 + 2 t^4 \\[-.8ex] + 3 t^3 + 3 t^2 + 5 t + 4 \bp\end{aligned} & t^2 + 2 t + 1 \\[0ex]
  \sbmatrix{5&0\\0&3} & \sbmatrix{5&0\\0&3} & \sbmatrix{2&1\\4&6} & \sbmatrix{1&6\\1&0} & \sbmatrix{5&4\\0&3} & \begin{aligned}\tp t^8 + 6 t^7 + 5 t^6 + t^5 + 3 t^4 \\[-.8ex] + t^3 + 5 t^2 + 6 t + 1 \bp\end{aligned} & t^2 + 6 t + 1 \\[0ex]
  \sbmatrix{0&4\\5&2} & \sbmatrix{5&1\\5&4} & \sbmatrix{3&4\\6&6} & \sbmatrix{1&1\\0&1} & \sbmatrix{6&2\\5&3} & \begin{aligned}\tp t^8 + 3 t^6 + 6 t^4 + 3 t^2 + 1 \bp\end{aligned} & t^2 + 5 t + 1 \\[0ex]
  \sbmatrix{4&0\\0&2} & \sbmatrix{6&6\\1&0} & \sbmatrix{4&3\\0&2} & \sbmatrix{2&3\\0&4} & \sbmatrix{5&6\\3&1} & \begin{aligned}\tp t^8 + 2 t^6 + 3 t^4 + 2 t^2 + 1 \bp\end{aligned} & t^2 + t + 1 \\[0ex]
  \sbmatrix{4&0\\0&2} & \sbmatrix{4&0\\0&2} & \sbmatrix{2&6\\0&4} & \sbmatrix{1&3\\6&5} & \sbmatrix{2&0\\3&4} & \begin{aligned}\tp t^8 + 4 t^7 + 3 t^6 + 2 t^5 + 5 t^4 \\[-.8ex] + 2 t^3 + 3 t^2 + 4 t + 1 \bp\end{aligned} & t^2 + t + 1 \\[0ex]
  \sbmatrix{4&0\\0&2} & \sbmatrix{4&0\\0&2} & \sbmatrix{1&3\\6&5} & \sbmatrix{0&6\\1&6} & \sbmatrix{4&0\\3&2} & \begin{aligned}\tp t^8 + t^7 + 5 t^6 + 6 t^5 + 3 t^4 \\[-.8ex] + 6 t^3 + 5 t^2 + t + 1 \bp\end{aligned} & t^2 + t + 1 \\[0ex]
  \sbmatrix{0&4\\5&2} & \sbmatrix{3&1\\3&6} & \sbmatrix{2&1\\6&0} & \sbmatrix{4&4\\3&5} & \sbmatrix{3&4\\6&6} & \begin{aligned}\tp 4 t^8 + 6 t^7 + 4 t^6 + 6 t^5 + 2 t^4 \\[-.8ex] + 6 t^3 + 4 t^2 + 6 t + 4 \bp\end{aligned} & t^2 + 5 t + 1 \\[0ex]
  \sbmatrix{0&4\\5&2} & \sbmatrix{0&4\\5&2} & \sbmatrix{0&4\\5&2} & \sbmatrix{0&4\\5&2} & \sbmatrix{0&4\\5&2} & \begin{aligned}\tp t^8 + 2 t^7 + 2 t^6 + 3 t^5 + 6 t^4 \\[-.8ex] + 3 t^3 + 2 t^2 + 2 t + 1 \bp\end{aligned} & t^2 + 5 t + 1 \\[0ex]
  \sbmatrix{0&6\\1&4} & \sbmatrix{1&5\\6&3} & \sbmatrix{4&4\\5&0} & \sbmatrix{1&2\\1&3} & \sbmatrix{4&3\\2&0} & \begin{aligned}\tp 6 t^6 + 4 t^4 + t^3 + 4 t^2 + 6 \bp\end{aligned} & t^2 + 3 t + 1 \\[0ex]
  \sbmatrix{5&0\\0&3} & \sbmatrix{5&0\\1&3} & \sbmatrix{3&0\\6&5} & \sbmatrix{5&0\\0&3} & \sbmatrix{6&4\\1&2} & \begin{aligned}\tp 0 \bp\end{aligned} & t^2 + 6 t + 1 \\[0ex]
  \sbmatrix{4&0\\0&2} & \sbmatrix{4&0\\0&2} & \sbmatrix{4&0\\0&2} & \sbmatrix{4&0\\0&2} & \sbmatrix{4&0\\0&2} & \begin{aligned}\tp t^8 + 6 t^7 + 4 t^6 + 5 t^5 + 5 t^3 \\[-.8ex] + 4 t^2 + 6 t + 1 \bp\end{aligned} & t^2 + t + 1 \\[0ex]
  \sbmatrix{0&4\\5&2} & \sbmatrix{4&2\\6&5} & \sbmatrix{1&0\\3&1} & \sbmatrix{0&2\\3&2} & \sbmatrix{0&2\\3&2} & \begin{aligned}\tp t^4 + 3 t^3 + 6 t^2 + 3 t + 1 \bp\end{aligned} & t^2 + 5 t + 1 \\[0ex]
  \sbmatrix{5&0\\0&3} & \sbmatrix{5&0\\0&3} & \sbmatrix{3&0\\2&5} & \sbmatrix{2&4\\1&6} & \sbmatrix{3&2\\0&5} & \begin{aligned}\tp t^8 + 3 t^7 + 3 t^6 + 5 t^5 + 5 t^4 \\[-.8ex] + 5 t^3 + 3 t^2 + 3 t + 1 \bp\end{aligned} & t^2 + 6 t + 1 \\[0ex]
  \sbmatrix{4&0\\0&2} & \sbmatrix{5&2\\2&1} & \sbmatrix{5&2\\2&1} & \sbmatrix{5&1\\4&1} & \sbmatrix{5&2\\2&1} & \begin{aligned}\tp 4 t^8 + 4 t^6 + t^5 + 2 t^4 + t^3 \\[-.8ex] + 4 t^2 + 4 \bp\end{aligned} & t^2 + t + 1 \\[0ex]
  \sbmatrix{5&0\\0&3} & \sbmatrix{6&5\\5&2} & \sbmatrix{4&4\\2&4} & \sbmatrix{0&2\\3&1} & \sbmatrix{5&3\\0&3} & \begin{aligned}\tp 4 t^8 + t^7 + 4 t^6 + t^5 + 2 t^4 \\[-.8ex] + t^3 + 4 t^2 + t + 4 \bp\end{aligned} & t^2 + 6 t + 1 \\[0ex]
  \sbmatrix{6&0\\0&6} & \sbmatrix{6&0\\0&6} & \sbmatrix{6&0\\0&6} & \sbmatrix{6&0\\0&6} & \sbmatrix{6&0\\0&6} & \begin{aligned}\tp t^8 + 5 t^7 + 2 t^6 + 4 t^5 + 6 t^4 \\[-.8ex] + 4 t^3 + 2 t^2 + 5 t + 1 \bp\end{aligned} & t^2 + 2 t + 1 \\[0ex]
  \sbmatrix{0&6\\1&3} & \sbmatrix{1&2\\4&2} & \sbmatrix{1&1\\1&2} & \sbmatrix{6&6\\5&4} & \sbmatrix{1&4\\2&2} & \begin{aligned}\tp t^6 + 3 t^4 + t^3 + 3 t^2 + 1 \bp\end{aligned} & t^2 + 4 t + 1 \\[0ex]
  \sbmatrix{0&4\\5&5} & \sbmatrix{3&2\\6&2} & \sbmatrix{4&2\\5&1} & \sbmatrix{0&1\\6&5} & \sbmatrix{0&1\\6&5} & \begin{aligned}\tp t^4 + 4 t^3 + 6 t^2 + 4 t + 1 \bp\end{aligned} & t^2 + 2 t + 1 \\[0ex]
  \sbmatrix{0&4\\5&5} & \sbmatrix{6&2\\0&6} & \sbmatrix{6&1\\0&6} & \sbmatrix{2&2\\6&3} & \sbmatrix{5&2\\3&0} & \begin{aligned}\tp t^8 + 3 t^6 + 6 t^4 + 3 t^2 + 1 \bp\end{aligned} & t^2 + 2 t + 1 \\[0ex]
  \sbmatrix{4&0\\0&2} & \sbmatrix{5&2\\2&1} & \sbmatrix{4&0\\0&2} & \sbmatrix{5&1\\4&1} & \sbmatrix{5&4\\1&1} & \begin{aligned}\tp 2 t^4 + 4 t^3 + 6 t^2 + 4 t + 2 \bp\end{aligned} & t^2 + t + 1 \\[0ex]
  \sbmatrix{0&4\\5&5} & \sbmatrix{0&4\\5&5} & \sbmatrix{4&2\\5&1} & \sbmatrix{5&4\\5&0} & \sbmatrix{6&4\\0&6} & \begin{aligned}\tp t^8 + 6 t^7 + 3 t^6 + 2 t^5 + t^4 \\[-.8ex] + 2 t^3 + 3 t^2 + 6 t + 1 \bp\end{aligned} & t^2 + 2 t + 1 \\[0ex]
  \sbmatrix{0&6\\1&4} & \sbmatrix{5&6\\6&6} & \sbmatrix{1&5\\6&3} & \sbmatrix{0&4\\5&4} & \sbmatrix{4&4\\5&0} & \begin{aligned}\tp t^8 + 3 t^7 + 6 t^6 + 6 t^5 + 4 t^4 \\[-.8ex] + 6 t^3 + 6 t^2 + 3 t + 1 \bp\end{aligned} & t^2 + 3 t + 1 \\[0ex]
  \sbmatrix{4&0\\0&2} & \sbmatrix{6&6\\1&0} & \sbmatrix{2&0\\0&4} & \sbmatrix{0&5\\4&6} & \sbmatrix{4&0\\1&2} & \begin{aligned}\tp 2 t^8 + t^7 + 6 t^6 + 4 t^5 + 3 t^4 \\[-.8ex] + 4 t^3 + 6 t^2 + t + 2 \bp\end{aligned} & t^2 + t + 1 \\[0ex]
  \sbmatrix{0&6\\1&3} & \sbmatrix{2&1\\1&1} & \sbmatrix{1&3\\5&2} & \sbmatrix{3&5\\4&0} & \sbmatrix{5&4\\6&5} & \begin{aligned}\tp t^8 + 4 t^7 + 6 t^6 + t^5 + 4 t^4 \\[-.8ex] + t^3 + 6 t^2 + 4 t + 1 \bp\end{aligned} & t^2 + 4 t + 1 \\[0ex]
  \sbmatrix{0&6\\1&3} & \sbmatrix{0&6\\1&3} & \sbmatrix{2&1\\1&1} & \sbmatrix{2&5\\3&1} & \sbmatrix{5&1\\3&5} & \begin{aligned}\tp t^8 + 2 t^4 + 1 \bp\end{aligned} & t^2 + 4 t + 1 \\[0ex]
  \sbmatrix{0&4\\5&2} & \sbmatrix{5&1\\5&4} & \sbmatrix{5&4\\3&4} & \sbmatrix{1&1\\0&1} & \sbmatrix{1&4\\0&1} & \begin{aligned}\tp 4 t^8 + 2 t^7 + 3 t^6 + 4 t^5 + 2 t^4 \\[-.8ex] + 4 t^3 + 3 t^2 + 2 t + 4 \bp\end{aligned} & t^2 + 5 t + 1 \\[0ex]
  \sbmatrix{5&0\\0&3} & \sbmatrix{0&6\\1&1} & \sbmatrix{5&0\\2&3} & \sbmatrix{5&0\\4&3} & \sbmatrix{0&4\\5&1} & \begin{aligned}\tp 4 t^8 + 6 t^7 + 6 t^6 + 4 t^5 + 4 t^4 \\[-.8ex] + 4 t^3 + 6 t^2 + 6 t + 4 \bp\end{aligned} & t^2 + 6 t + 1 \\[0ex]
  \sbmatrix{0&6\\1&3} & \sbmatrix{5&1\\3&5} & \sbmatrix{6&2\\1&4} & \sbmatrix{4&6\\5&6} & \sbmatrix{3&3\\2&0} & \begin{aligned}\tp 2 t^8 + 5 t^7 + 3 t^6 + 2 t^5 + 6 t^4 \\[-.8ex] + 2 t^3 + 3 t^2 + 5 t + 2 \bp\end{aligned} & t^2 + 4 t + 1 \\[0ex]
  \sbmatrix{4&0\\0&2} & \sbmatrix{5&2\\2&1} & \sbmatrix{5&1\\4&1} & \sbmatrix{5&1\\4&1} & \sbmatrix{4&0\\0&2} & \begin{aligned}\tp 2 t^8 + 4 t^6 + 6 t^4 + 4 t^2 + 2 \bp\end{aligned} & t^2 + t + 1 \\[0ex]
  \sbmatrix{5&0\\0&3} & \sbmatrix{6&5\\5&2} & \sbmatrix{6&5\\5&2} & \sbmatrix{6&3\\6&2} & \sbmatrix{6&5\\5&2} & \begin{aligned}\tp 4 t^8 + 4 t^6 + 6 t^5 + 2 t^4 + 6 t^3 \\[-.8ex] + 4 t^2 + 4 \bp\end{aligned} & t^2 + 6 t + 1 \\[0ex]
  \sbmatrix{0&6\\1&0} & \sbmatrix{2&6\\5&5} & \sbmatrix{4&6\\3&3} & \sbmatrix{4&6\\3&3} & \sbmatrix{0&1\\6&0} & \begin{aligned}\tp 4 t^4 + t^2 + 4 \bp\end{aligned} & t^2 + 1 \\[0ex]
  \sbmatrix{0&4\\5&2} & \sbmatrix{1&0\\3&1} & \sbmatrix{0&2\\3&2} & \sbmatrix{1&2\\0&1} & \sbmatrix{3&4\\6&6} & \begin{aligned}\tp 4 t^8 + 5 t^6 + 3 t^4 + 5 t^2 + 4 \bp\end{aligned} & t^2 + 5 t + 1 \\[0ex]
  \sbmatrix{4&0\\0&2} & \sbmatrix{5&2\\2&1} & \sbmatrix{3&5\\3&3} & \sbmatrix{6&4\\5&0} & \sbmatrix{4&0\\4&2} & \begin{aligned}\tp 4 t^8 + 6 t^7 + 4 t^6 + 6 t^5 + 2 t^4 \\[-.8ex] + 6 t^3 + 4 t^2 + 6 t + 4 \bp\end{aligned} & t^2 + t + 1 \\[0ex]
  \sbmatrix{0&6\\1&0} & \sbmatrix{2&6\\5&5} & \sbmatrix{3&5\\5&4} & \sbmatrix{1&4\\3&6} & \sbmatrix{5&3\\3&2} & \begin{aligned}\tp 6 t^8 + 5 t^7 + t^6 + t^5 + 4 t^4 \\[-.8ex] + t^3 + t^2 + 5 t + 6 \bp\end{aligned} & t^2 + 1 \\[0ex]
  \sbmatrix{5&0\\0&3} & \sbmatrix{5&0\\0&3} & \sbmatrix{5&0\\0&3} & \sbmatrix{5&0\\0&3} & \sbmatrix{5&0\\0&3} & \begin{aligned}\tp t^8 + t^7 + 4 t^6 + 2 t^5 + 2 t^3 \\[-.8ex] + 4 t^2 + t + 1 \bp\end{aligned} & t^2 + 6 t + 1 \\[0ex]
  \sbmatrix{0&4\\5&2} & \sbmatrix{1&0\\3&1} & \sbmatrix{5&1\\5&4} & \sbmatrix{1&0\\6&1} & \sbmatrix{2&2\\3&0} & \begin{aligned}\tp 4 t^4 + 5 t^3 + 3 t^2 + 5 t + 4 \bp\end{aligned} & t^2 + 5 t + 1 \\[0ex]
  \sbmatrix{4&0\\0&2} & \sbmatrix{6&6\\1&0} & \sbmatrix{4&5\\0&2} & \sbmatrix{4&3\\0&2} & \sbmatrix{6&2\\3&0} & \begin{aligned}\tp 4 t^8 + t^7 + 6 t^6 + 3 t^5 + 4 t^4 \\[-.8ex] + 3 t^3 + 6 t^2 + t + 4 \bp\end{aligned} & t^2 + t + 1 \\[0ex]
  \sbmatrix{0&6\\1&4} & \sbmatrix{0&6\\1&4} & \sbmatrix{5&6\\6&6} & \sbmatrix{4&6\\1&0} & \sbmatrix{1&6\\5&3} & \begin{aligned}\tp t^8 + 2 t^4 + 1 \bp\end{aligned} & t^2 + 3 t + 1 \\[0ex]
  \sbmatrix{5&0\\0&3} & \sbmatrix{6&5\\5&2} & \sbmatrix{5&0\\0&3} & \sbmatrix{6&3\\6&2} & \sbmatrix{6&6\\3&2} & \begin{aligned}\tp 2 t^4 + 3 t^3 + 6 t^2 + 3 t + 2 \bp\end{aligned} & t^2 + 6 t + 1 \\[0ex]
  \sbmatrix{5&0\\0&3} & \sbmatrix{0&6\\1&1} & \sbmatrix{3&0\\0&5} & \sbmatrix{1&3\\2&0} & \sbmatrix{5&6\\0&3} & \begin{aligned}\tp 2 t^8 + 6 t^7 + 6 t^6 + 3 t^5 + 3 t^4 \\[-.8ex] + 3 t^3 + 6 t^2 + 6 t + 2 \bp\end{aligned} & t^2 + 6 t + 1 \\[0ex]
  \sbmatrix{5&0\\0&3} & \sbmatrix{5&0\\1&3} & \sbmatrix{3&1\\0&5} & \sbmatrix{1&4\\5&0} & \sbmatrix{3&3\\0&5} & \begin{aligned}\tp t^8 + 2 t^7 + t^6 + 2 t^5 + 4 t^4 \\[-.8ex] + 2 t^3 + t^2 + 2 t + 1 \bp\end{aligned} & t^2 + 6 t + 1 \\[0ex]
  \sbmatrix{0&4\\5&5} & \sbmatrix{6&0\\5&6} & \sbmatrix{1&4\\6&4} & \sbmatrix{4&4\\6&1} & \sbmatrix{6&1\\0&6} & \begin{aligned}\tp 4 t^8 + t^7 + 4 t^6 + t^5 + 2 t^4 \\[-.8ex] + t^3 + 4 t^2 + t + 4 \bp\end{aligned} & t^2 + 2 t + 1 \\[0ex]
  \sbmatrix{0&6\\1&3} & \sbmatrix{2&4\\2&1} & \sbmatrix{0&5\\4&3} & \sbmatrix{1&4\\2&2} & \sbmatrix{1&2\\4&2} & \begin{aligned}\tp 2 t^8 + 6 t^7 + 4 t^6 + t^4 + 4 t^2 \\[-.8ex] + 6 t + 2 \bp\end{aligned} & t^2 + 4 t + 1 \\[0ex]
  \sbmatrix{0&6\\1&4} & \sbmatrix{0&6\\1&4} & \sbmatrix{0&6\\1&4} & \sbmatrix{0&6\\1&4} & \sbmatrix{0&6\\1&4} & \begin{aligned}\tp t^8 + 4 t^7 + t^6 + 5 t^5 + 5 t^3 \\[-.8ex] + t^2 + 4 t + 1 \bp\end{aligned} & t^2 + 3 t + 1 \\[0ex]
  \sbmatrix{4&0\\0&2} & \sbmatrix{6&6\\1&0} & \sbmatrix{5&5\\5&1} & \sbmatrix{3&3\\5&3} & \sbmatrix{4&0\\4&2} & \begin{aligned}\tp 2 t^8 + 6 t^7 + 4 t^6 + t^4 + 4 t^2 \\[-.8ex] + 6 t + 2 \bp\end{aligned} & t^2 + t + 1 \\[0ex]
  \sbmatrix{5&0\\0&3} & \sbmatrix{0&6\\1&1} & \sbmatrix{6&2\\2&2} & \sbmatrix{4&2\\4&4} & \sbmatrix{5&3\\0&3} & \begin{aligned}\tp 2 t^8 + t^7 + 4 t^6 + t^4 + 4 t^2 \\[-.8ex] + t + 2 \bp\end{aligned} & t^2 + 6 t + 1 \\[0ex]
  \sbmatrix{5&0\\0&3} & \sbmatrix{6&5\\5&2} & \sbmatrix{6&3\\6&2} & \sbmatrix{6&3\\6&2} & \sbmatrix{5&0\\0&3} & \begin{aligned}\tp 2 t^8 + 4 t^6 + 6 t^4 + 4 t^2 + 2 \bp\end{aligned} & t^2 + 6 t + 1 \\[0ex]
  \sbmatrix{4&0\\0&2} & \sbmatrix{4&3\\0&2} & \sbmatrix{2&4\\0&4} & \sbmatrix{4&0\\0&2} & \sbmatrix{5&3\\6&1} & \begin{aligned}\tp 0 \bp\end{aligned} & t^2 + t + 1 \\[0ex]
  \sbmatrix{0&6\\1&4} & \sbmatrix{2&4\\6&2} & \sbmatrix{1&6\\5&3} & \sbmatrix{1&3\\3&3} & \sbmatrix{6&4\\2&5} & \begin{aligned}\tp 2 t^8 + 2 t^7 + 3 t^6 + 5 t^5 + 6 t^4 \\[-.8ex] + 5 t^3 + 3 t^2 + 2 t + 2 \bp\end{aligned} & t^2 + 3 t + 1 \\[0ex]
  \sbmatrix{0&4\\5&5} & \sbmatrix{4&2\\5&1} & \sbmatrix{0&1\\6&5} & \sbmatrix{5&4\\5&0} & \sbmatrix{6&1\\0&6} & \begin{aligned}\tp 4 t^8 + 5 t^6 + 3 t^4 + 5 t^2 + 4 \bp\end{aligned} & t^2 + 2 t + 1 \\[0ex]
  \sbmatrix{0&6\\1&0} & \sbmatrix{2&6\\5&5} & \sbmatrix{5&3\\3&2} & \sbmatrix{0&5\\4&0} & \sbmatrix{3&5\\5&4} & \begin{aligned}\tp 6 t^8 + 2 t^7 + t^6 + 6 t^5 + 4 t^4 \\[-.8ex] + 6 t^3 + t^2 + 2 t + 6 \bp\end{aligned} & t^2 + 1 \\[0ex]
  \sbmatrix{0&6\\1&3} & \sbmatrix{0&6\\1&3} & \sbmatrix{0&6\\1&3} & \sbmatrix{0&6\\1&3} & \sbmatrix{0&6\\1&3} & \begin{aligned}\tp t^8 + 3 t^7 + t^6 + 2 t^5 + 2 t^3 \\[-.8ex] + t^2 + 3 t + 1 \bp\end{aligned} & t^2 + 4 t + 1 \\[0ex]
\end{longtable}
}

\bibliographystyle{amsalpha}
\renewcommand{\MR}[1]{}
\bibliography{research}

\end{document}